\documentclass[12pt]{article}

\usepackage{amsfonts,amsmath,amsthm,amssymb,amscd}
\usepackage[dvips]{graphicx}
\usepackage{mathrsfs}

\theoremstyle{definition}
\newtheorem{theorem}{Theorem}
\newtheorem{lemma}[theorem]{Lemma}
\newtheorem{proposition}[theorem]{Proposition}

\numberwithin{equation}{section}
\numberwithin{theorem}{section}

\newcommand{\Ker}{\operatorname{Ker}}
\newcommand{\sign}{\operatorname{sign}}

\begin{document}

\begin{center}
{\bf{\Large Analogue of the theta group $\Gamma_{\theta},$ II }}
\end{center}

\begin{center}
By Kazuhide Matsuda
\end{center}

\begin{center}
Faculty of Fundamental Science, \\
National Institute of Technology (KOSEN), Niihama College,\\
7-1 Yagumo-chou, Niihama, Ehime, Japan, 792-8580. \\
E-mail: ka.matsuda@niihama-nct.ac.jp  \\
Fax: 81-0897-37-7809 
\end{center}

\noindent
{\bf Abstract}
In this series of papers, we introduce higher level versions of the theta group $\Gamma_{\theta}.$ 
In this paper, we treat the theta group of level $5$, $\Gamma_{\theta,5},$ 
and 
construct modular forms on $\Gamma_{\theta,5}$.  
Moreover 
we compute their multiplier systems. 
For this purpose, 
we derive transformation formula 
of theta function with characteristics. 

\section{Introduction}
\label{intro}
The upper half plane $\mathscr{H}$ is defined by 
$
\mathscr{H}=
\{
\tau\in\mathbb{C} \, | \,  \Im \tau>0
\}. 
$
Moreover, set $q=\exp(2\pi i  \tau)$ and define 
the Dedekind $\eta$ function and the theta constant by 
\begin{equation*}
\displaystyle
\eta(\tau)=q^{\frac{1}{24}} \prod_{n=1}^{\infty} (1-q^n) 
\,\,\mathrm{and} \,\,
\theta(\tau)=\sum_{n=-\infty}^{\infty}  e^{\pi i n^2 \tau}.  
\end{equation*}
The modular group $\Gamma(1)$ and the theta group $\Gamma_{\theta}$ are defined by 
\begin{equation*}
\Gamma(1)=SL(2,\mathbb{Z}) \,\,
\mathrm{and} \,\,
\Gamma_{\theta}=
\left\{
\begin{pmatrix}
a & b \\
c & d
\end{pmatrix}
\in
\Gamma(1) \,\Big| \,
a\equiv d  \bmod{2} \,\,
\mathrm{and} \,\,
b \equiv c  \bmod{2} 
\right\}.
\end{equation*}
\par
It is well known that 
$\eta(\tau)$ is a modular form of weight $\displaystyle 1/2$ on $\Gamma(1)$ and 
$\theta(\tau)$ is a modular form of weight $\displaystyle 1/2$ on $\Gamma_{\theta}.$ 
In \cite{Matsuda1}, 
from the multiplier systems, $\nu_{\eta}, \nu_{\theta},$ of $\eta(\tau)$ and $\theta(\tau)$, 
we obtain many examples of modular groups. 
\par
In \cite{Matsuda2}, we  proposed the theta group of level $N$ in the following way: 
\begin{equation*}
\Gamma_{\theta,N}
=
\left\{
\begin{pmatrix}
a & b \\
c & d
\end{pmatrix}
\in
\Gamma(1) \,\Big| \,
a\equiv d  \bmod{N} \,\,
\mathrm{and} \,\,
b \equiv -c  \bmod{N} 
\right\}, \,\,
(N=1,2,3,\ldots).
\end{equation*}
In addition, we constructed modular forms on $\Gamma_{\theta,3}$ and $\Gamma_{\theta,4}$ and computed their multiplier systems. 
For this purpose, we used $\eta(\tau).$ 
\par
In this paper, we treat the theta group of level $5$:
\begin{equation*}
\Gamma_{\theta,5}
=
\left\{
\begin{pmatrix}
a & b \\
c & d
\end{pmatrix}
\in
\Gamma(1) \,\Big| \,
\begin{pmatrix}
a & b \\
c & d
\end{pmatrix}
\equiv 
\pm
\begin{pmatrix}
1 &0 \\
0 & 1
\end{pmatrix}, \,
\pm
\begin{pmatrix}
0 &-1 \\
1 & 0
\end{pmatrix}
\bmod{5}
\right\}.  \,\,
\end{equation*}
In order to construct modular forms on $\Gamma_{\theta,5}$,  
we use Farkas and Kra's theory of theta function with characteristics. 
\par
The aim of this paper 
is to 
construct modular forms on $\Gamma_{\theta,5}$ and compute their multiplier systems. 
Section \ref{sec:eta} 
describes the multiplier system $\nu_{\eta}$ of $\eta(\tau)$. 
Section \ref{sec:theta-function} treats Farkas and Kra's theory of theta function and derives transformation formula of theta function with characteristics. 
Sections \ref{sec:modularforms-(1/5,1/5),(1/5,9/5)}, \ref{sec:modularforms-(3/5,3/5),(3/5,7/5)} and \ref{sec:level5-eta-product}  
construct modular forms on $\Gamma_{\theta,5}$ and compute their multiplier systems. 
Moreover by considering the kernels of them,  
we discover some examples of modular groups. 
Section \ref{sec:coset} 
considers 
coset decomposition of $\Gamma(1)$ modulo $\Gamma_{\theta,5}$, 
which yields parabolic points and elleptic elements of $\Gamma_{\theta,5}$.


\section{The multiplier system of $\eta(\tau)$}
\label{sec:eta}

Knopp \cite[pp.51]{Knopp} proved that for each 
$
M
=
\begin{pmatrix}
a & b \\
c & d
\end{pmatrix}
\in
\Gamma(1),
$
\begin{equation}
\label{eqn-eta-multiplier}
\nu_{\eta}(M)=
\begin{cases}
\displaystyle
\left(
\frac{d}{c}
\right)^{*}   
\exp
\left\{
\frac{\pi i}{12}
[
(a+d)c-bd(c^2-1)-3c
]
\right\} 
&
\text{if $c$ is odd,}
\vspace{2mm}  \\
\displaystyle
\left(
\frac{c}{d}
\right)_{*}
\exp
\left\{
\frac{\pi i}{12}
[
(a+d)c-bd(c^2-1)+3d-3-3cd
]
\right\} 
&
\text{if $c$ is even,}
\end{cases}
\end{equation}
where for $c,d\in \mathbb{Z}$ with $(c,d)=1$ and $d\equiv 1 \bmod{2},$ 
\begin{equation*}
\left(
\frac{c}{d}
\right)^{*}=
\left(
\frac{c}{ |d|}
\right)
\mathrm{and} 
\left(
\frac{c}{d}
\right)_{*}
=
\left(
\frac{c}{ |d|}
\right)
(-1)^{
\frac{\sign c-1}{2}
\frac{\sign d-1}{2}
}. 
\end{equation*}

\section{Farkas and Kra's theory of theta functions}
\label{sec:theta-function}

\subsection{Definitions}
Following the work of Farkas and Kra \cite{Farkas-Kra}, 
we introduce the  theta function with characteristics, 
which is defined by 
\begin{align*}
\theta 
\left[
\begin{array}{c}
\epsilon \\
\epsilon^{\prime}
\end{array}
\right] (v, \tau) 
=
\theta 
\left[
\begin{array}{c}
\epsilon \\
\epsilon^{\prime}
\end{array}
\right] (v) 
:=&\sum_{n\in\mathbb{Z}} \exp
\left(2\pi i\left[ \frac12\left(n+\frac{\epsilon}{2}\right)^2 \tau+\left(n+\frac{\epsilon}{2}\right)\left(v+\frac{\epsilon^{\prime}}{2}\right) \right] \right), 
\end{align*}
where $\epsilon, \epsilon^{\prime}\in\mathbb{R}, \, v \in\mathbb{C},$ and $\tau\in\mathscr{H}.$ 
The theta constants are given by 
\begin{equation*}
\theta 
\left[
\begin{array}{c}
\epsilon \\
\epsilon^{\prime}
\end{array}
\right]
:=
\theta 
\left[
\begin{array}{c}
\epsilon \\
\epsilon^{\prime}
\end{array}
\right] (0, \tau).
\end{equation*}
In terms of Jacobi's original $\vartheta$ notation for theta functions,
\begin{equation}
\label{eqn:theta-null}
\vartheta_2=
\theta 
\left[
\begin{array}{c}
1 \\
0
\end{array}
\right], \,\,
\vartheta_3=
\theta 
\left[
\begin{array}{c}
0 \\
0
\end{array}
\right]=\theta(\tau), \,\,
\vartheta_4=
\theta 
\left[
\begin{array}{c}
0 \\
1
\end{array}
\right]. 
\end{equation}
Furthermore, 
we denote the derivative coefficient of the theta function by 
\begin{equation*}
\theta^{\prime} 
\left[
\begin{array}{c}
\epsilon \\
\epsilon^{\prime}
\end{array}
\right]
:=\left.
\frac{\partial}{\partial v} 
\theta 
\left[
\begin{array}{c}
\epsilon \\
\epsilon^{\prime}
\end{array}
\right] (v, \tau)
\right|_{v=0}. 
\end{equation*}
In particular, Jacobi's derivative formula is given by 
\begin{equation}
\label{eqn:Jacobi-derivative}
\theta^{\prime} 
\left[
\begin{array}{c}
1 \\
1
\end{array}
\right] 
=
-\pi 
\theta
\left[
\begin{array}{c}
0 \\
0
\end{array}
\right] 
\theta
\left[
\begin{array}{c}
1 \\
0
\end{array}
\right] 
\theta
\left[
\begin{array}{c}
0 \\
1
\end{array}
\right].  
\end{equation}

\subsection{Basic properties}
We first note that 
for $m,n\in\mathbb{Z},$ 
\begin{equation}
\label{eqn:integer-char}
\theta 
\left[
\begin{array}{c}
\epsilon \\
\epsilon^{\prime}
\end{array}
\right] (v+n+m\tau, \tau) =
\exp(2\pi i)\left[\frac{n\epsilon-m\epsilon^{\prime}}{2}-mv-\frac{m^2\tau}{2}\right]
\theta 
\left[
\begin{array}{c}
\epsilon \\
\epsilon^{\prime}
\end{array}
\right] (v,\tau),
\end{equation}
and 
\begin{equation}
\label{eqn:char-even}
\theta 
\left[
\begin{array}{c}
\epsilon +2m\\
\epsilon^{\prime}+2n
\end{array}
\right] 
(v,\tau)
=\exp(\pi i \epsilon n)
\theta 
\left[
\begin{array}{c}
\epsilon \\
\epsilon^{\prime}
\end{array}
\right] 
(v,\tau).
\end{equation}
Furthermore, 
it is easy to see that 
\begin{equation}
\label{eqn:character-minus}
\theta 
\left[
\begin{array}{c}
-\epsilon \\
-\epsilon^{\prime}
\end{array}
\right] (v,\tau)
=
\theta 
\left[
\begin{array}{c}
\epsilon \\
\epsilon^{\prime}
\end{array}
\right] (-v,\tau)
\,\,
\mathrm{and}
\,\,
\theta^{\prime} 
\left[
\begin{array}{c}
-\epsilon \\
-\epsilon^{\prime}
\end{array}
\right] (v,\tau)
=
-
\theta^{\prime} 
\left[
\begin{array}{c}
\epsilon \\
\epsilon^{\prime}
\end{array}
\right] (-v,\tau).
\end{equation}
\par
For $m,n\in\mathbb{R},$ 
we see that 
\begin{align}
\label{eqn:real-char}
&\theta 
\left[
\begin{array}{c}
\epsilon \\
\epsilon^{\prime}
\end{array}
\right] \left(v+\frac{m\tau +n }{2}, \tau\right)   \notag\\
&=
\exp(2\pi i)\left[
-\frac{mv}{2}-\frac{m^2\tau}{8}-\frac{m(\epsilon^{\prime}+n)}{4}
\right]
\theta 
\left[
\begin{array}{c}
\epsilon+m \\
\epsilon^{\prime}+n
\end{array}
\right] 
(v,\tau). 
\end{align}
We note that 
$\theta 
\left[
\begin{array}{c}
\epsilon \\
\epsilon^{\prime}
\end{array}
\right] \left(v, \tau\right)$ has only one zero in the fundamental parallelogram, 
which is given by 
$$
v=\frac{1-\epsilon}{2}\tau+\frac{1-\epsilon^{\prime}}{2}. 
$$

\subsection{Jacobi's triple product identity}
All the theta functions have infinite product expansions, which are given by 
\begin{align}
\theta 
\left[
\begin{array}{c}
\epsilon \\
\epsilon^{\prime}
\end{array}
\right] (v, \tau) &=\exp\left(\frac{\pi i \epsilon \epsilon^{\prime}}{2}\right) x^{\frac{\epsilon^2}{4}} z^{\frac{\epsilon}{2}}    \notag  \\
                           &\quad 
                           \displaystyle \times\prod_{n=1}^{\infty}(1-x^{2n})(1+e^{\pi i \epsilon^{\prime}} x^{2n-1+\epsilon} z)(1+e^{-\pi i \epsilon^{\prime}} x^{2n-1-\epsilon}/z),  \label{eqn:Jacobi-triple}
\end{align}
where $x=\exp(\pi i \tau)$ and $z=\exp(2\pi i v).$ 
Therefore, it follows from Jacobi's derivative formula (\ref{eqn:Jacobi-derivative}) that 
\begin{equation*}
\label{eqn:Jacobi}
\theta^{\prime} 
\left[
\begin{array}{c}
1 \\
1
\end{array}
\right](0,\tau) 
=
-2\pi 
q^{\frac18}
\prod_{n=1}^{\infty}(1-q^n)^3 
=
-2 \pi 
\eta^3(\tau), \,\,q=\exp(2\pi i \tau). 
\end{equation*}

\subsection{Transformation formula}

\begin{theorem}
\label{thm:transformation-(1,1)}
{\it
For each 
$
M
=
\begin{pmatrix}
a & b \\
c & d
\end{pmatrix}
\in
\Gamma(1),
$ 
we have 
\begin{equation}
\label{eqn:transformation-(1,1)}
\theta 
\left[
\begin{array}{c}
1 \\
1
\end{array}
\right] \left(\frac{v}{c\tau+d},  \frac{a\tau+b}{c \tau+d} \right)
=
\nu_{\eta}^3(M)
\cdot
(c\tau+d)^{\frac12}
\cdot
\exp(\pi i)
\left[
\frac{cv^2}{c\tau+d}
\right] 
\cdot
\theta
\left[
\begin{array}{c}
1 \\
1
\end{array}
\right](v,\tau).  
\end{equation}
}
\end{theorem}

\begin{proof}
Since $ad-bc=1,$ it follows that
\begin{equation*}
1=-c(a\tau+b)+a(c\tau+d), \,\, \mathrm{and} \,\,
\tau=d(a\tau+b)-b(c\tau+d).
\end{equation*}
Equation (\ref{eqn:integer-char}) yields 
\begin{align*}
&
\theta 
\left[
\begin{array}{c}
1 \\
1
\end{array}
\right] \left(\frac{v+1}{c\tau+d},  \frac{a\tau+b}{c \tau+d} \right) 
= \, 
\theta 
\left[
\begin{array}{c}
1 \\
1
\end{array}
\right] \left(\frac{v -c(a\tau+b)+a(c\tau+d) }{c\tau+d},  \frac{a\tau+b}{c \tau+d} \right)  \\
=& \,
\theta 
\left[
\begin{array}{c}
1 \\
1
\end{array}
\right] 
\left(
\frac{v }{c\tau+d}               
-c\frac{a\tau+b}{c\tau+d}+a, 
\frac{a\tau+b}{c \tau+d} 
\right)  \\
=& \,
\exp(2 \pi i)
\left[
\frac12
\left\{
(1+a)(1-c)-1+2c
\right\}
+
\frac{cv }{c\tau+d} 
+
\frac{c }{ 2( c\tau+d )} 
\right]
\cdot
\theta 
\left[
\begin{array}{c}
1 \\
1
\end{array}
\right] 
\left(\frac{v}{c\tau+d},  \frac{a\tau+b}{c \tau+d} \right)  \\
=&\, 
\exp(2 \pi i)
\left[
\frac12+
\frac{cv }{c\tau+d} 
+
\frac{c }{ 2( c\tau+d )} 
\right]
\cdot
\theta 
\left[
\begin{array}{c}
1 \\
1
\end{array}
\right] 
\left(\frac{v}{c\tau+d},  \frac{a\tau+b}{c \tau+d} \right), 
\end{align*}
and
\begin{align*}
&
\theta 
\left[
\begin{array}{c}
1 \\
1
\end{array}
\right] \left(\frac{v+\tau}{c\tau+d},  \frac{a\tau+b}{c \tau+d} \right) 
= \, 
\theta 
\left[
\begin{array}{c}
1 \\
1
\end{array}
\right] \left(\frac{v +d(a\tau+b)-b(c\tau+d) }{c\tau+d},  \frac{a\tau+b}{c \tau+d} \right)  \\
=&\, 
\theta 
\left[
\begin{array}{c}
1 \\
1
\end{array}
\right]
\left(\frac{v  }{c\tau+d}+d\frac{a\tau+b}{c\tau+d}-b,  \frac{a\tau+b}{c \tau+d} \right)  \\
=& \, 
\exp(2 \pi i)
\left[
-
\frac12
\left\{
(1+b)(1+d)-1
\right\}
-
\frac{d v }{c\tau+d} 
-
\frac{d \tau }{ 2( c\tau+d )} 
\right]
\theta 
\left[
\begin{array}{c}
1 \\
1
\end{array}
\right] \left(\frac{v}{c\tau+d},  \frac{a\tau+b}{c \tau+d} \right)   \\
=& \,
\exp(2 \pi i)
\left[
\frac12
-
\frac{d v }{c\tau+d} 
-
\frac{d \tau }{ 2( c\tau+d )} 
\right]
\theta 
\left[
\begin{array}{c}
1 \\
1
\end{array}
\right] \left(\frac{v}{c\tau+d},  \frac{a\tau+b}{c \tau+d} \right).  
\end{align*}   
Moreover, 
equation (\ref{eqn:integer-char}) implies that 
\begin{align*}
&
\exp(\pi i)
\left[
\frac{c(v+1)^2}{c\tau+d}
\right] 
\theta 
\left[
\begin{array}{c}
1 \\
1
\end{array}
\right] 
(v+1, \tau)=
\exp(\pi i)
\left[
\frac{c v^2}{c\tau+d}
+
\frac{2cv}{c\tau+d}
+\frac{c}{c\tau+d}
+
1
\right]
\theta 
\left[
\begin{array}{c}
1 \\
1
\end{array}
\right] 
(v, \tau)  \\
=& \, 
\exp(2 \pi i)
\left[
\frac12+
\frac{cv}{c\tau+d}
+\frac{c}{ 2(c\tau+d)}
\right]
\exp(\pi i)
\left[
\frac{c v^2}{c\tau+d}
\right] 
\theta 
\left[
\begin{array}{c}
1 \\
1
\end{array}
\right] 
(v, \tau),
\end{align*}
and
\begin{align*}
&
\exp(\pi i)
\left[
\frac{c(v+\tau)^2}{c\tau+d}
\right] 
\theta 
\left[
\begin{array}{c}
1 \\
1
\end{array}
\right] 
(v+\tau, \tau) \\
=&
\exp(\pi i)
\left[
-1
-
\frac{2dv}{c\tau+d}
-
\frac{d\tau}{c\tau+d}
\right]
\exp(\pi i)
\left[
\frac{c v^2}{c\tau+d}
\right] 
\theta 
\left[
\begin{array}{c}
1 \\
1
\end{array}
\right] 
(v, \tau)  \\
=& \,
\exp(2 \pi i)
\left[
\frac12
-
\frac{dv}{c\tau+d}
-
\frac{d\tau}{ 2(c\tau+d) }
\right] 
\exp(\pi i)
\left[
\frac{cv^2}{c\tau+d}
\right] 
\theta 
\left[
\begin{array}{c}
1 \\
1
\end{array}
\right] 
(v, \tau). 
\end{align*}
\par
We note that  
$
\theta 
\left[
\begin{array}{c}
1 \\
1
\end{array}
\right] 
(v, \tau)
$ 
has 
only one zero $v=0$ 
in the fundamental parallelogram, 
which implies that
the function, 
\begin{equation*}
F(v)=
\frac{
\theta 
\left[
\begin{array}{c}
1 \\
1
\end{array}
\right] \left(\frac{v}{c\tau+d},  \frac{a\tau+b}{c \tau+d} \right) 
}
{
\exp(\pi i)
\left[
\frac{c v^2}{c\tau+d}
\right] 
\theta 
\left[
\begin{array}{c}
1 \\
1
\end{array}
\right] 
(v, \tau)
}
\end{equation*}
is an elliptic function with the period $1$ and $\tau,$ and is holomorphic. 
Therefore it follows that 
$F(v)$ is a constant, 
which implies that 
\begin{align*}
&
\frac{
\theta 
\left[
\begin{array}{c}
1 \\
1
\end{array}
\right] \left(\frac{v}{c\tau+d},  \frac{a\tau+b}{c \tau+d} \right) 
}
{
\exp(\pi i)
\left[
\frac{c v^2}{c\tau+d}
\right] 
\theta 
\left[
\begin{array}{c}
1 \\
1
\end{array}
\right] 
(v, \tau)
}
=
\lim_{v\to 0} 
\frac{
\theta 
\left[
\begin{array}{c}
1 \\
1
\end{array}
\right] \left(\frac{v}{c\tau+d},  \frac{a\tau+b}{c \tau+d} \right) 
}
{
\exp(\pi i)
\left[
\frac{c v^2}{c\tau+d}
\right] 
\theta 
\left[
\begin{array}{c}
1 \\
1
\end{array}
\right] 
(v, \tau)
} 
=
\frac{
\theta^{\prime} 
\left[
\begin{array}{c}
1 \\
1
\end{array}
\right] 
\left(
0, \frac{a\tau+b}{c\tau+d}
\right)
}
{
\theta^{\prime} 
\left[
\begin{array}{c}
1 \\
1
\end{array}
\right] 
\left(
0, \tau
\right)
} 
\cdot
\frac{1}{c\tau+d}
 \\
=&
\frac
{
-2 \pi 
\eta^3
\left(
\frac{a\tau+b}{c\tau+d}
\right)
}
{
-2\pi
\eta^3(\tau)
} 
\cdot
\frac{1}{c\tau+d}
=
\left\{
\nu_{\eta}(M) (c\tau+d)^{\frac12}
\right\}^3
\cdot
\frac{1}{c\tau+d}
=
\nu_{\eta}^3(M) 
(c\tau+d)^{\frac12}.
\end{align*}
\end{proof}

\begin{theorem}
\label{thm:transformation-general}
{\it
For each 
$
M
=
\begin{pmatrix}
a & b \\
c & d
\end{pmatrix}
\in
\Gamma(1),
$ 
we have 
\begin{align*}
&
\theta 
\left[
\begin{array}{c}
\epsilon \\
\epsilon^{\prime}
\end{array}
\right] \left(\frac{v}{c\tau+d},  \frac{a\tau+b}{c \tau+d} \right)  \notag \\
=&
\nu_{\eta}^3(M)
\cdot
(c\tau+d)^{\frac12}
\cdot
\exp(\pi i)
\left[
\frac{cv^2}{c\tau+d}
\right] 
\cdot
\theta
\left[
\begin{array}{c}
a\epsilon+c\epsilon^{\prime}-ac \\
b\epsilon+d\epsilon^{\prime}-bd
\end{array}
\right](v,\tau)  \notag  \\ 
&\times
\exp(\pi i)
\Bigg[
\left(
a\cdot\frac{1-\epsilon}{2}+c\cdot\frac{1-\epsilon^{\prime}}{2}
\right)
(b+d-bd)  
+
\frac{ (b-1)(d-1) }{2}   -\frac{1-\epsilon}{2}    \notag \\
&
\hspace{25mm}
-b\cdot \frac{1-\epsilon}{2}
\left(
a\cdot\frac{1-\epsilon}{2}+c\cdot\frac{1-\epsilon^{\prime}}{2}
\right)
-c
\cdot
\frac{1-\epsilon^{\prime}}{2} 
\left(
b\cdot\frac{1-\epsilon}{2}+d\cdot\frac{1-\epsilon^{\prime}}{2}
\right)  
\,
\Bigg]. 
\end{align*}
}
\end{theorem}

\begin{proof}
We substitute 
$
\displaystyle
v
+
\frac{-1+\epsilon}{2}(a\tau+b)
+
\frac{-1+\epsilon^{\prime} }{2}(c\tau+d)
$
into 
both sides of equation (\ref{eqn:transformation-(1,1)}). 
By equation (\ref{eqn:real-char}), we have 
\begin{align*}
&
\theta 
\left[
\begin{array}{c}
1 \\
1
\end{array}
\right] \left(\frac{v
+
\frac{-1+\epsilon}{2}(a\tau+b)
+
\frac{-1+\epsilon^{\prime} }{2} ( c\tau+d)}{c\tau+d},  \frac{a\tau+b}{c \tau+d} \right)  \\
=&\, 
\theta 
\left[
\begin{array}{c}
1 \\
1
\end{array}
\right] 
\left(
\frac{v}{c\tau+d}   
+
\frac{-1+\epsilon}{2}
\cdot
\frac{a\tau+b}{c \tau+d}
+
\frac{-1+\epsilon^{\prime} }{2}, 
 \frac{a\tau+b}{c \tau+d}
\right) \\
=& \, 
\exp(\pi i)
\left[
-\frac14(-1+\epsilon)^2 \cdot 
\frac{a\tau+b}{c \tau+d}
-
( -1+\epsilon^{\prime}) \cdot
\frac{v}{c\tau+d} 
-
\frac12 
(-1+\epsilon)
\epsilon^{\prime}
\right] 
\cdot
\theta 
\left[
\begin{array}{c}
\epsilon \\
\epsilon^{\prime}
\end{array}
\right] 
\left(
\frac{v}{c\tau+d}, 
\frac{a\tau+b}{c \tau+d} 
\right),
\end{align*} 
and 
\begin{align*}
&
\exp(\pi i)
\left[
\frac{c}{c\tau+d} 
\left\{
v
+
\frac{-1+\epsilon}{2}(a\tau+b)
+
\frac{-1+\epsilon^{\prime} }{2}(c\tau+d)
\right\}^2
\right]  \\
&\times
\theta 
\left[
\begin{array}{c}
1 \\
1
\end{array}
\right] 
\left(
v
+
\frac{-1+\epsilon}{2}(a\tau+b)
+
\frac{-1+\epsilon^{\prime} }{2}(c\tau+d), 
\tau
\right) \\
=&
\exp (\pi i) 
\Bigg[
\frac{cv^2}{c\tau+d} 
+
\frac{  (1-\epsilon)^2}{4} ab 
-
\frac{(1-\epsilon)^2}{4} 
\frac{a\tau+b}{c \tau+d} 
+
\frac{(1-\epsilon^{\prime})^2}{4} cd
-
\frac{(-1+\epsilon)v}{c\tau+d}  \\
&\hspace{15mm}
+
\frac{(1-\epsilon) (1-\epsilon^{\prime}) }{2} bc 
-
\frac12
\left\{
a(\epsilon-1)+c(\epsilon^{\prime}-1)
\right\}
\cdot
\left\{
b(\epsilon-1)+d(\epsilon^{\prime}-1)+1
\right\}
\Bigg] \\
&\times
\theta 
\left[
\begin{array}{c}
a\epsilon+c\epsilon^{\prime}-a-c+1 \\
b\epsilon+d\epsilon^{\prime}-b-d+1
\end{array}
\right] 
(v,\tau),  
\end{align*}
which implies that 
\begin{align*}
&
\theta 
\left[
\begin{array}{c}
\epsilon \\
\epsilon^{\prime}
\end{array}
\right] \left(\frac{v}{c\tau+d},  \frac{a\tau+b}{c \tau+d} \right)   \\
=&
\nu_{\eta}^3(M)
\cdot
(c\tau+d)^{\frac12}
\cdot
\exp(\pi i)
\left[
\frac{cv^2}{c\tau+d}
\right]
\cdot
\theta 
\left[
\begin{array}{c}
a\epsilon+c\epsilon^{\prime}-a-c+1 \\
b\epsilon+d\epsilon^{\prime}-b-d+1
\end{array}
\right] 
(v,\tau)  \\
&\times
\exp(\pi i) 
\left[
-
\frac{  (1-\epsilon)^2}{4} ab 
-
\frac{(1-\epsilon^{\prime})^2}{4} cd
-
\frac12 (1-\epsilon)
-
\frac{(1-\epsilon) (1-\epsilon^{\prime}) }{2} bc 
-
\frac12
\left\{
a(\epsilon-1)+c(\epsilon^{\prime}-1)
\right\}
\right]. 
\end{align*}
\par
Since $ad-bc=1,$ it follows that 
\begin{equation*}
(1-a)(1-c)\equiv(1-b)(1-d) \equiv 0 \bmod{2}, 
\end{equation*}
which implies that 
\begin{align*}
a\epsilon+c\epsilon^{\prime}-a-c+1
=& \,
a\epsilon+c\epsilon^{\prime}
-
ac
+
2
\cdot 
\frac12
(1-a)(1-c),  \\
b\epsilon+d\epsilon^{\prime}-b-d+1
=& \,
b\epsilon+d\epsilon^{\prime}
-
bd
+
2
\cdot 
\frac12
(1-b)(1-d). 
\end{align*}
By equation (\ref{eqn:char-even}), 
we have 
\begin{align*}
&
\theta 
\left[
\begin{array}{c}
\epsilon \\
\epsilon^{\prime}
\end{array}
\right] \left(\frac{v}{c\tau+d},  \frac{a\tau+b}{c \tau+d} \right)   \\
=&
\nu_{\eta}^3(M)
\cdot
(c\tau+d)^{\frac12}
\cdot
\exp(\pi i)
\left[
\frac{cv^2}{c\tau+d}
\right]
\cdot 
\theta 
\left[
\begin{array}{c}
a\epsilon+c\epsilon^{\prime}-ac  \\
b\epsilon+d\epsilon^{\prime}-bd
\end{array}
\right] 
(v,\tau) 
\cdot
\exp(\pi i)[ E],
\end{align*} 
where 
\begin{align*}
E=&
(a\epsilon+c\epsilon^{\prime}-ac  )\cdot 
\frac12
(1-b)(1-d) \\
& 
-
\frac{  (1-\epsilon)^2}{4} ab 
-
\frac{(1-\epsilon^{\prime})^2}{4} cd
-
\frac12 (1-\epsilon)
-
\frac{(1-\epsilon) (1-\epsilon^{\prime}) }{2} bc 
-
\frac12
\left\{
a(\epsilon-1)+c(\epsilon^{\prime}-1)
\right\}  \\
=&
\{
a(\epsilon-1)
+
c(\epsilon^{\prime}-1)
-
(1-a)(1-c)+1
\}
\cdot
\frac12
(1-b)(1-d) 
-
\frac12 (1-\epsilon)
\\
&
+
\frac{b}{4}
(1-\epsilon)
\{
a(\epsilon-1)
+
c(\epsilon^{\prime}-1)
\} 
+
\frac{c}{4}
(1-\epsilon^{\prime})
\{
b(\epsilon-1)
+
d(\epsilon^{\prime}-1)
\}  \\
&
\hspace{90mm}
-
\frac12
\{
a(\epsilon-1)
+
c(\epsilon^{\prime}-1)
\}   \\
\equiv & 
\{
a(\epsilon-1)
+
c(\epsilon^{\prime}-1)
\}
\cdot
\frac12
(1-b)(1-d)
+
\frac12
(1-b)(1-d)
-
\frac12 (1-\epsilon)
\\
&
+
\frac{b}{4}
(1-\epsilon)
\{
a(\epsilon-1)
+
c(\epsilon^{\prime}-1)
\} 
+
\frac{c}{4}
(1-\epsilon^{\prime})
\{
b(\epsilon-1)
+
d(\epsilon^{\prime}-1)
\}  \\
&\hspace{80mm}
-
\frac12
\{
a(\epsilon-1)
+
c(\epsilon^{\prime}-1)
\}    \bmod{2} \\
=& \, 
\left(
a\cdot\frac{1-\epsilon}{2}+c\cdot\frac{1-\epsilon^{\prime}}{2}
\right)
(b+d-bd)  
+
\frac{ (b-1)(d-1) }{2} -\frac{1-\epsilon}{2}   \notag \\
&
\hspace{25mm}
-b\cdot \frac{1-\epsilon}{2}
\left(
a\cdot\frac{1-\epsilon}{2}+c\cdot\frac{1-\epsilon^{\prime}}{2}
\right)
-c
\cdot
\frac{1-\epsilon^{\prime}}{2} 
\left(
b\cdot\frac{1-\epsilon}{2}+d\cdot\frac{1-\epsilon^{\prime}}{2}
\right),  
\end{align*}
which proves the theorem. 
\end{proof}

\section{Modular forms on $\Gamma_{\theta,5}$ (1) }
\label{sec:modularforms-(1/5,1/5),(1/5,9/5)}

\subsection{The case where 
$
M
\equiv
\begin{pmatrix}
1 & 0 \\
0 & 1
\end{pmatrix} 
\bmod{5}. 
$ 
}

\begin{theorem}
\label{thm:(1/5,1/5),(1/5,9/5)-I}
{\it
Set 
$
M
=
\begin{pmatrix}
a & b \\
c & d
\end{pmatrix}
\in
\Gamma_{\theta,5}
$ 
with 
$
M
\equiv
\begin{pmatrix}
1 & 0 \\
0 & 1
\end{pmatrix} 
\bmod{5}. 
$ 
Then the multiplier system of 
$
\theta 
\left[
\begin{array}{c}
\frac15 \vspace{1mm} \\
\frac15
\end{array}
\right] \left(0, \tau  \right) 
\theta 
\left[
\begin{array}{c}
\frac15 \vspace{1mm} \\
\frac95
\end{array}
\right] \left(0, \tau  \right) 
$ 
is given by 
\begin{equation*}
\nu(M)
=
\nu_{\eta}^6(M)
\exp
\left[
\frac{\pi i}{5}
\left(
-
\frac{4b}{5}
-
\frac{8ab}{5}
-
\frac{8cd}{5}
\right)
\right]. 
\end{equation*}
}
\end{theorem}

\begin{proof}
Since $ad-bc=1,$ 
it follows that 
\begin{equation*}
(a,c)\equiv (1,1),(1,0),(0,1) \bmod{2} \, \mathrm{and} \,
(b,d)\equiv (1,1),(1,0),(0,1) \bmod{2},
\end{equation*}
which implies that 
\begin{align*}
&
\frac{a}{5}+\frac{c}{5}-ac-\frac15 =
\frac{a-1}{5}+\frac{c}{5}-ac \equiv 0 \bmod{2},  \\
&
\frac{b}{5}+\frac{d}{5}-bd-\frac15 =
\frac{b}{5}+\frac{d-1}{5}-bd \equiv 0 \bmod{2},  \\
&
\frac{a}{5}+\frac{9c}{5}-ac-\frac15 =
\frac{a-1}{5}+\frac{9c}{5}-ac \equiv 0 \bmod{2},   \\
&
\frac{b}{5}+\frac{9d}{5}-bd-\frac95 =
\frac{b}{5}+\frac{9(d-1)}{5}-bd \equiv 0 \bmod{2}. 
\end{align*} 
By equation (\ref{eqn:char-even}) and 
Theorem \ref{thm:transformation-general}, 
we find that 
\begin{align*}
&
\theta 
\left[
\begin{array}{c}
\frac15 \vspace{1mm} \\
\frac15
\end{array}
\right] \left(0, M\tau  \right) 
\theta 
\left[
\begin{array}{c}
\frac15 \vspace{1mm} \\
\frac95
\end{array}
\right] \left(0, M\tau  \right) \\
=&
\nu_{\eta}^6(M)
\cdot
(c\tau+d) \cdot
\theta 
\left[
\begin{array}{c}
a/5+c/5-ac \vspace{1mm} \\
b/5+d/5-bd
\end{array}
\right] \left(0, \tau  \right) 
\theta 
\left[
\begin{array}{c}
a/5+9c/5-ac \vspace{1mm} \\
b/5+9d/5-bd
\end{array}
\right] \left(0, \tau  \right) \\
&
\times
\exp
(\pi i)
\left[
\frac{4a}{5}
(b+d-bd)
-
\frac{8}{25}
ab
-
\frac{8}{25}
cd
-\frac45
\right]  \\
=&
\nu_{\eta}^6(M)
\cdot
(c\tau+d) \cdot
\theta 
\left[
\begin{array}{c}
\frac15 \vspace{1mm} \\
\frac15
\end{array}
\right] \left(0, \tau  \right) 
\theta 
\left[
\begin{array}{c}
\frac15 \vspace{1mm} \\
\frac95
\end{array}
\right] \left(0, \tau  \right) \\
&
\times
\exp(\pi i)
\left[
\frac15 \cdot \frac12
\left(
\frac{b}{5}
+\frac{d-1}{5}
-bd
\right)
\right] 
\cdot
\exp(\pi i)
\left[
\frac15 \cdot \frac12
\left(
\frac{b}{5}
+\frac{9(d-1)}{5}
-bd
\right)
\right]   \\
&
\times
\exp
(\pi i)
\left[
\frac{4a}{5}
(b+d-bd)
-
\frac{8}{25}
ab
-
\frac{8}{25}
cd
-\frac45
\right]  \\
=&
\nu_{\eta}^6(M)
\cdot
(c\tau+d) \cdot
\theta 
\left[
\begin{array}{c}
\frac15 \vspace{1mm} \\
\frac15
\end{array}
\right] \left(0, \tau  \right) 
\theta 
\left[
\begin{array}{c}
\frac15 \vspace{1mm} \\
\frac95
\end{array}
\right] \left(0, \tau  \right)
\cdot
\exp(\pi i) 
\left[
E
\right],
\end{align*}
where 
\begin{align*}
E=&
\frac15 \cdot \frac12
\left(
\frac{b}{5}
+\frac{d-1}{5}
-bd
\right) 
+
\frac15 \cdot \frac12
\left(
\frac{b}{5}
+\frac{9(d-1)}{5}
-bd
\right)
+
\frac{4a}{5}
(b+d-bd) \\
&\hspace{100mm}
-
\frac{8}{25}
ab
-
\frac{8}{25}
cd
-\frac45   \\
=&
\frac{1}{25}
(
b+5d-5-5bd
)
+
\frac{4a}{5}
(b+d-bd)
-
\frac{8}{25}ab
-
\frac{8}{25}
cd
-
\frac45  \\
=&
\frac{2}{25}
\left\{
- 
\frac52 (b-1)(d-1)-2b
+10(a-1)
-
10a(b-1)(d-1)
-4ab
-4cd
\right\} \\
\equiv&
\frac{2}{25}
(
-2b-4ab-4cd
)
\bmod{2}   \\
=&
\frac15
\left(
-\frac{4b}{5}
-
\frac{8ab}{5}
-
\frac{8cd}{5}
\right), 
\end{align*}
which proves the theorem. 
\end{proof}

\subsection{The case where 
$
M
\equiv
\begin{pmatrix}
-1 & 0 \\
0 & -1
\end{pmatrix} 
\bmod{5}. 
$ 
}

\begin{theorem}
\label{thm:(1/5,1/5),(1/5,9/5)-minusI}
{\it
Set 
$
M
=
\begin{pmatrix}
a & b \\
c & d
\end{pmatrix}
\in
\Gamma_{\theta,5}
$ 
with 
$
M
\equiv
\begin{pmatrix}
-1 & 0 \\
0 & -1
\end{pmatrix} 
\bmod{5}. 
$ 
Then the multiplier system of 
$
\theta 
\left[
\begin{array}{c}
\frac15 \vspace{1mm} \\
\frac15
\end{array}
\right] \left(0, \tau  \right) 
\theta 
\left[
\begin{array}{c}
\frac15 \vspace{1mm} \\
\frac95
\end{array}
\right] \left(0, \tau  \right) 
$ 
is given by 
\begin{equation*}
\nu(M)
=
\nu_{\eta}^6(M)
\exp
\left[
\frac{\pi i}{5}
\left(
-
\frac{6b}{5}
-
\frac{8ab}{5}
-
\frac{8cd}{5}
\right)
\right]. 
\end{equation*}
}
\end{theorem}

\begin{proof}
Since $ad-bc=1,$ 
it follows that 
\begin{equation*}
(a,c)\equiv (1,1),(1,0),(0,1) \bmod{2} \, \mathrm{and} \,
(b,d)\equiv (1,1),(1,0),(0,1) \bmod{2},
\end{equation*}
which implies that 
\begin{align*}
&
\frac{a}{5}+\frac{c}{5}-ac-\left(-\frac15\right) =
\frac{a+1}{5}+\frac{c}{5}-ac \equiv 0 \bmod{2}, \\
&
\frac{b}{5}+\frac{d}{5}-bd-\left(-\frac15 \right) =
\frac{b}{5}+\frac{d+1}{5}-bd \equiv 0 \bmod{2},  \\
&
\frac{a}{5}+\frac{9c}{5}-ac-\left(-\frac15\right) =
\frac{a+1}{5}+\frac{9c}{5}-ac \equiv 0 \bmod{2},  \\
&
\frac{b}{5}+\frac{9d}{5}-bd-\left(-\frac95\right) =
\frac{b}{5}+\frac{9(d+1)}{5}-bd 
\equiv 0 \bmod{2}. 
\end{align*} 
By equations (\ref{eqn:char-even}), (\ref{eqn:character-minus}) and 
Theorem \ref{thm:transformation-general}, 
we find that 
\begin{align*}
&
\theta 
\left[
\begin{array}{c}
\frac15 \vspace{1mm} \\
\frac15
\end{array}
\right] \left(0, M\tau  \right) 
\theta 
\left[
\begin{array}{c}
\frac15 \vspace{1mm} \\
\frac95
\end{array}
\right] \left(0, M\tau  \right) \\
=&
\nu_{\eta}^6(M)
\cdot
(c\tau+d) \cdot
\theta 
\left[
\begin{array}{c}
a/5+c/5-ac \vspace{1mm} \\
b/5+d/5-bd
\end{array}
\right] \left(0, \tau  \right) 
\theta 
\left[
\begin{array}{c}
a/5+9c/5-ac \vspace{1mm} \\
b/5+9d/5-bd
\end{array}
\right] \left(0, \tau  \right) \\
&
\times
\exp
(\pi i)
\left[
\frac{4a}{5}
(b+d-bd)
-
\frac{8}{25}
ab
-
\frac{8}{25}
cd
-\frac45
\right]  \\
=&
\nu_{\eta}^6(M)
\cdot
(c\tau+d) \cdot
\theta 
\left[
\begin{array}{c}
\frac15 \vspace{1mm} \\
\frac15
\end{array}
\right] \left(0, \tau  \right) 
\theta 
\left[
\begin{array}{c}
\frac15 \vspace{1mm} \\
\frac95
\end{array}
\right] \left(0, \tau  \right) \\
&
\times
\exp(\pi i)
\left[
-
\frac15 \cdot \frac12
\left(
\frac{b}{5}
+\frac{d+1}{5}
-bd
\right)
\right] 
\cdot
\exp(\pi i)
\left[
-
\frac15 \cdot \frac12
\left(
\frac{b}{5}
+\frac{9(d+1)}{5}
-bd
\right)
\right]   \\
&
\times
\exp
(\pi i)
\left[
\frac{4a}{5}
(b+d-bd)
-
\frac{8}{25}
ab
-
\frac{8}{25}
cd
-\frac45
\right]  \\
=&
\nu_{\eta}^6(M)
\cdot
(c\tau+d) \cdot
\theta 
\left[
\begin{array}{c}
\frac15 \vspace{1mm} \\
\frac15
\end{array}
\right] \left(0, \tau  \right) 
\theta 
\left[
\begin{array}{c}
\frac15 \vspace{1mm} \\
\frac95
\end{array}
\right] \left(0, \tau  \right)
\cdot
\exp(\pi i) 
\left[
E
\right],
\end{align*}
where 
\begin{align*}
E=&
-
\frac15 \cdot \frac12
\left(
\frac{b}{5}
+\frac{d+1}{5}
-bd
\right) 
-
\frac15 \cdot \frac12
\left(
\frac{b}{5}
+\frac{9(d+1)}{5}
-bd
\right)
+
\frac{4a}{5}
(b+d-bd) \\
&\hspace{90mm}
-
\frac{8}{25}
ab
-
\frac{8}{25}
cd
-\frac45   \\
=&
\frac{1}{25}
(
-b-5d-5+5bd
)
+
\frac{4a}{5}
(b+d-bd)
-
\frac{8}{25}ab
-
\frac{8}{25}
cd
-
\frac45  \\
=&
\frac{2}{25}
\left\{
\frac52 (b-1)(d+1)-3b
-10(a+1)
-
10a(b-1)(d+1)
+20ab
-4ab
-4cd
\right\} \\
\equiv&
\frac{2}{25}
(
-3b-4ab-4cd
)
\bmod{2}   \\
=&
\frac15
\left(
-\frac{6b}{5}
-
\frac{8ab}{5}
-
\frac{8cd}{5}
\right), 
\end{align*}
which proves the theorem. 
\end{proof}

\subsection{The case where 
$
M
\equiv
\begin{pmatrix}
0 & -1 \\
1 & 0
\end{pmatrix} 
\bmod{5}. 
$ 
}

\begin{theorem}
\label{thm:(1/5,1/5),(1/5,9/5)-T}
{\it
Set 
$
M
=
\begin{pmatrix}
a & b \\
c & d
\end{pmatrix}
\in
\Gamma_{\theta,5}
$ 
with 
$
M
\equiv
\begin{pmatrix}
0 & -1 \\
1 & 0
\end{pmatrix} 
\bmod{5}. 
$ 
Then the multiplier system of 
$
\theta 
\left[
\begin{array}{c}
\frac15 \vspace{1mm} \\
\frac15
\end{array}
\right] \left(0, \tau  \right) 
\theta 
\left[
\begin{array}{c}
\frac15 \vspace{1mm} \\
\frac95
\end{array}
\right] \left(0, \tau  \right) 
$ 
is given by 
\begin{equation*}
\nu(M)
=
\nu_{\eta}^6(M)
\exp
\left[
\frac{\pi i}{5}
\left(
-
5
-
\frac{4d}{5}
-
\frac{8ab}{5}
-
\frac{8cd}{5}
\right)
\right]. 
\end{equation*}
}
\end{theorem}

\begin{proof}
Since $ad-bc=1,$ 
it follows that 
\begin{equation*}
(a,c)\equiv (1,1),(1,0),(0,1) \bmod{2} \, \mathrm{and} \,
(b,d)\equiv (1,1),(1,0),(0,1) \bmod{2},
\end{equation*}
which implies that 
\begin{align*}
&
\frac{a}{5}+\frac{c}{5}-ac-\frac15 =
\frac{a}{5}+\frac{c-1}{5}-ac \equiv 0 \bmod{2}, \\
&
\frac{b}{5}+\frac{d}{5}-bd-\frac95 =
\frac{b-9}{5}+\frac{d}{5}-bd \equiv 0 \bmod{2},  \\
&
\frac{a}{5}+\frac{9c}{5}-ac-\left(-\frac15 \right) =
\frac{a}{5}+\frac{9c+1}{5}-ac \equiv 0 \bmod{2}, \\
&
\frac{b}{5}+\frac{9d}{5}-bd-\left(-\frac15 \right)=
\frac{b+1}{5}+\frac{9d}{5}-bd  
\equiv 0 \bmod{2}. 
\end{align*} 
By equations (\ref{eqn:char-even}), (\ref{eqn:character-minus}) and 
Theorem \ref{thm:transformation-general}, 
we find that 
\begin{align*}
&
\theta 
\left[
\begin{array}{c}
\frac15 \vspace{1mm} \\
\frac15
\end{array}
\right] \left(0, M\tau  \right) 
\theta 
\left[
\begin{array}{c}
\frac15 \vspace{1mm} \\
\frac95
\end{array}
\right] \left(0, M\tau  \right) \\
=&
\nu_{\eta}^6(M)
\cdot
(c\tau+d) \cdot
\theta 
\left[
\begin{array}{c}
a/5+c/5-ac \vspace{1mm} \\
b/5+d/5-bd
\end{array}
\right] \left(0, \tau  \right) 
\theta 
\left[
\begin{array}{c}
a/5+9c/5-ac \vspace{1mm} \\
b/5+9d/5-bd
\end{array}
\right] \left(0, \tau  \right) \\
&
\times
\exp
(\pi i)
\left[
\frac{4a}{5}
(b+d-bd)
-
\frac{8}{25}
ab
-
\frac{8}{25}
cd
-\frac45
\right]  \\
=&
\nu_{\eta}^6(M)
\cdot
(c\tau+d) \cdot
\theta 
\left[
\begin{array}{c}
\frac15 \vspace{1mm} \\
\frac15
\end{array}
\right] \left(0, \tau  \right) 
\theta 
\left[
\begin{array}{c}
\frac15 \vspace{1mm} \\
\frac95
\end{array}
\right] \left(0, \tau  \right) \\
&
\times
\exp(\pi i)
\left[
\frac15 \cdot \frac12
\left(
\frac{b-9}{5}
+\frac{d}{5}
-bd
\right)
\right] 
\cdot
\exp(\pi i)
\left[
-
\frac15 \cdot \frac12
\left(
\frac{b+1}{5}
+\frac{9d}{5}
-bd
\right)
\right]   \\
&
\times
\exp
(\pi i)
\left[
\frac{4a}{5}
(b+d-bd)
-
\frac{8}{25}
ab
-
\frac{8}{25}
cd
-\frac45
\right]  \\
=&
\nu_{\eta}^6(M)
\cdot
(c\tau+d) \cdot
\theta 
\left[
\begin{array}{c}
\frac15 \vspace{1mm} \\
\frac15
\end{array}
\right] \left(0, \tau  \right) 
\theta 
\left[
\begin{array}{c}
\frac15 \vspace{1mm} \\
\frac95
\end{array}
\right] \left(0, \tau  \right)
\cdot
\exp(\pi i) 
\left[
E
\right],
\end{align*}
where 
\begin{align*}
E=&
\frac15 \cdot \frac12
\left(
\frac{b-9}{5}
+\frac{d}{5}
-bd
\right) 
-
\frac15 \cdot \frac12
\left(
\frac{b+1}{5}
+\frac{9d}{5}
-bd
\right)
+
\frac{4a}{5}
(b+d-bd) \\
&\hspace{90mm}
-
\frac{8}{25}
ab
-
\frac{8}{25}
cd
-\frac45   \\
\equiv&
-
1
-
\frac{4d}{25}
-
\frac{8ab}{25}
-
\frac{8cd}{25} \bmod{2}  \\
=&
\frac15
\left[
-5
-
\frac{4d}{5}
-
\frac{8ab}{5}
-
\frac{8cd}{5}
\right], 
\end{align*}
which proves the theorem. 
\end{proof}

\subsection{The case where 
$
M
\equiv
\begin{pmatrix}
0 & 1 \\
-1 & 0
\end{pmatrix} 
\bmod{5}. 
$ 
}

\begin{theorem}
\label{thm:(1/5,1/5),(1/5,9/5)-minusT}
{\it
Set 
$
M
=
\begin{pmatrix}
a & b \\
c & d
\end{pmatrix}
\in
\Gamma_{\theta,5}
$ 
with 
$
M
\equiv
\begin{pmatrix}
0 & 1 \\
-1 & 0
\end{pmatrix} 
\bmod{5}. 
$ 
Then the multiplier system of 
$
\theta 
\left[
\begin{array}{c}
\frac15 \vspace{1mm} \\
\frac15
\end{array}
\right] \left(0, \tau  \right) 
\theta 
\left[
\begin{array}{c}
\frac15 \vspace{1mm} \\
\frac95
\end{array}
\right] \left(0, \tau  \right) 
$ 
is given by 
\begin{equation*}
\nu(M)
=
\nu_{\eta}^6(M)
\exp
\left[
\frac{\pi i}{5}
\left(
-
5
+
\frac{4d}{5}
-
\frac{8ab}{5}
-
\frac{8cd}{5}
\right)
\right]. 
\end{equation*}
}
\end{theorem}

\begin{proof}
Since $ad-bc=1,$ 
it follows that 
\begin{equation*}
(a,c)\equiv (1,1),(1,0),(0,1) \bmod{2} \, \mathrm{and} \,
(b,d)\equiv (1,1),(1,0),(0,1) \bmod{2},
\end{equation*}
which implies that 
\begin{align*}
&
\frac{a}{5}+\frac{c}{5}-ac-\left(-\frac15 \right) =
\frac{a}{5}+\frac{c+1}{5}-ac \equiv 0 \bmod{2}, \\
&
\frac{b}{5}+\frac{d}{5}-bd-\left(-\frac95\right) =
\frac{b+9}{5}+\frac{d}{5}-bd \equiv 0 \bmod{2},  \\
&
\frac{a}{5}+\frac{9c}{5}-ac-\frac15  =
\frac{a}{5}+\frac{9c-1}{5}-ac \equiv 0 \bmod{2}, \\
&
\frac{b}{5}+\frac{9d}{5}-bd-\frac15 =
\frac{b-1}{5}+\frac{9d}{5}-bd  
\equiv 0 \bmod{2}. 
\end{align*} 
By equations (\ref{eqn:char-even}), (\ref{eqn:character-minus}) and 
Theorem \ref{thm:transformation-general}, 
we find that 
\begin{align*}
&
\theta 
\left[
\begin{array}{c}
\frac15 \vspace{1mm} \\
\frac15
\end{array}
\right] \left(0, M\tau  \right) 
\theta 
\left[
\begin{array}{c}
\frac15 \vspace{1mm} \\
\frac95
\end{array}
\right] \left(0, M\tau  \right) \\
=&
\nu_{\eta}^6(M)
\cdot
(c\tau+d) \cdot
\theta 
\left[
\begin{array}{c}
a/5+c/5-ac \vspace{1mm} \\
b/5+d/5-bd
\end{array}
\right] \left(0, \tau  \right) 
\theta 
\left[
\begin{array}{c}
a/5+9c/5-ac \vspace{1mm} \\
b/5+9d/5-bd
\end{array}
\right] \left(0, \tau  \right) \\
&
\times
\exp
(\pi i)
\left[
\frac{4a}{5}
(b+d-bd)
-
\frac{8}{25}
ab
-
\frac{8}{25}
cd
-\frac45
\right]  \\
=&
\nu_{\eta}^6(M)
\cdot
(c\tau+d) \cdot
\theta 
\left[
\begin{array}{c}
\frac15 \vspace{1mm} \\
\frac15
\end{array}
\right] \left(0, \tau  \right) 
\theta 
\left[
\begin{array}{c}
\frac15 \vspace{1mm} \\
\frac95
\end{array}
\right] \left(0, \tau  \right) \\
&
\times
\exp(\pi i)
\left[
-
\frac15 \cdot \frac12
\left(
\frac{b+9}{5}
+\frac{d}{5}
-bd
\right)
\right] 
\cdot
\exp(\pi i)
\left[
\frac15 \cdot \frac12
\left(
\frac{b-1}{5}
+\frac{9d}{5}
-bd
\right)
\right]   \\
&
\times
\exp
(\pi i)
\left[
\frac{4a}{5}
(b+d-bd)
-
\frac{8}{25}
ab
-
\frac{8}{25}
cd
-\frac45
\right]  \\
=&
\nu_{\eta}^6(M)
\cdot
(c\tau+d) \cdot
\theta 
\left[
\begin{array}{c}
\frac15 \vspace{1mm} \\
\frac15
\end{array}
\right] \left(0, \tau  \right) 
\theta 
\left[
\begin{array}{c}
\frac15 \vspace{1mm} \\
\frac95
\end{array}
\right] \left(0, \tau  \right)
\cdot
\exp(\pi i) 
\left[
E
\right],
\end{align*}
where 
\begin{align*}
E=&
-
\frac15 \cdot \frac12
\left(
\frac{b+9}{5}
+\frac{d}{5}
-bd
\right) 
+
\frac15 \cdot \frac12
\left(
\frac{b-1}{5}
+\frac{9d}{5}
-bd
\right)
+
\frac{4a}{5}
(b+d-bd) \\
&\hspace{90mm}
-
\frac{8}{25}
ab
-
\frac{8}{25}
cd
-\frac45   \\
\equiv&
-
1
+
\frac{4d}{25}
-
\frac{8ab}{25}
-
\frac{8cd}{25} \bmod{2}  \\
=&
\frac15
\left[
-5
+
\frac{4d}{5}
-
\frac{8ab}{5}
-
\frac{8cd}{5}
\right], 
\end{align*}
which proves the theorem. 
\end{proof}

\subsection{Summary}

\begin{theorem}
\label{thm:(1/5,1/5),(1/5,9/5)-summary}
{\it
For each 
$
M
=
\begin{pmatrix}
a & b \\
c & d
\end{pmatrix}
\in
\Gamma_{\theta,5}, 
$ 
the multiplier system of 
$$
\theta 
\left[
\begin{array}{c}
\frac15 \vspace{1mm} \\
\frac15
\end{array}
\right] \left(0, \tau  \right) 
\theta 
\left[
\begin{array}{c}
\frac15 \vspace{1mm} \\
\frac95
\end{array}
\right] \left(0, \tau  \right) 
$$ 
is given by 
\begin{equation*}
\nu(M)
=
\nu_{\eta}^6(M)\cdot
\exp\left( \frac{\pi i}{5} E \right), 
\end{equation*}
where $E$ is defined by 
\begin{equation*}
E=
\begin{cases}
\displaystyle
-\frac{4b}{5}-\frac{8ab}{5}-\frac{8cd}{5} 
&
\text{
if 
$
M
\equiv 
\begin{pmatrix}
1 & 0 \\
0 & 1
\end{pmatrix}
\bmod{5}
$
} \\ 
\displaystyle
-\frac{6b}{5}-\frac{8ab}{5}-\frac{8cd}{5} 
&
\text{
if 
$
M
\equiv 
\begin{pmatrix}
-1 & 0 \\
0 & -1
\end{pmatrix}
\bmod{5}
$
} \\ 
\displaystyle
-5
-\frac{4d}{5}-\frac{8ab}{5}-\frac{8cd}{5} 
&
\text{
if 
$
M
\equiv 
\begin{pmatrix}
0 & -1 \\
1 & 0
\end{pmatrix}
\bmod{5}
$
} \\ 
\displaystyle
-5
+\frac{4d}{5}-\frac{8ab}{5}-\frac{8cd}{5} 
&
\text{
if 
$
M
\equiv 
\begin{pmatrix}
0 & 1 \\
-1 & 0
\end{pmatrix}
\bmod{5}. 
$
} 
\end{cases}
\end{equation*}
}
\end{theorem}

\begin{proof}
The theorem follows from 
Theorems 
\ref{thm:(1/5,1/5),(1/5,9/5)-I},  
\ref{thm:(1/5,1/5),(1/5,9/5)-minusI},  
\ref{thm:(1/5,1/5),(1/5,9/5)-T} and 
\ref{thm:(1/5,1/5),(1/5,9/5)-minusT}. 
\end{proof}

\begin{theorem}
\label{thm:(1/5,1/5),(1/5,9/5)-weight-2}
{\it
Set 
$
M
=
\begin{pmatrix}
a & b \\
c & d
\end{pmatrix}
\in
\Gamma_{\theta,5}
$ 
and 
\begin{equation*}
F(\tau)
=
\frac
{
\eta^6(\tau)
}
{
\theta 
\left[
\begin{array}{c}
\frac15 \vspace{1mm} \\
\frac15
\end{array}
\right] \left(0, \tau  \right) 
\theta 
\left[
\begin{array}{c}
\frac15 \vspace{1mm} \\
\frac95
\end{array}
\right] \left(0, \tau  \right) 
}.
\end{equation*}
Then, 
the multiplier system of $F$ is given by 
\begin{equation*}
\nu_{F}
=
\exp
\left[
\frac{\pi i}{5}
f(M)
\right], 
\end{equation*}
where $f(M)$ is defined by 
\begin{equation*}
f(M)=
\begin{cases}
\displaystyle
\frac{4b}{5}+\frac{8ab}{5}+\frac{8cd}{5} 
&
\text{
if 
$
M
\equiv 
\begin{pmatrix}
1 & 0 \\
0 & 1
\end{pmatrix}
\bmod{5}
$
} \\ 
\displaystyle
\frac{6b}{5}+\frac{8ab}{5}+\frac{8cd}{5} 
&
\text{
if 
$
M
\equiv 
\begin{pmatrix}
-1 & 0 \\
0 & -1
\end{pmatrix}
\bmod{5}
$
} \\ 
\displaystyle
5
+\frac{4d}{5}+\frac{8ab}{5}+\frac{8cd}{5} 
&
\text{
if 
$
M
\equiv 
\begin{pmatrix}
0 & -1 \\
1 & 0
\end{pmatrix}
\bmod{5}
$
} \\ 
\displaystyle
5
-\frac{4d}{5}+\frac{8ab}{5}+\frac{8cd}{5} 
&
\text{
if 
$
M
\equiv 
\begin{pmatrix}
0 & 1 \\
-1 & 0
\end{pmatrix}
\bmod{5}. 
$
} 
\end{cases}
\end{equation*}
}
\end{theorem}
\noindent
For each $k\in \mathbb{Z},$ 
we have 
\begin{equation*}
\nu_{F^k}(M)
=
\exp
\left[
\frac{\pi i k}{5}
f(M)
\right]. 
\end{equation*}

\subsection{The case where $k\equiv 0 \bmod{10}$  }

\begin{theorem}
\label{thm:k=0-(1/5,1/5)-(1/5,9/5)}
{\it
Suppose that $k\equiv 0 \bmod{10}.$ 
Then, it follows that 
\begin{equation*}
\Ker \nu_{F^k}=\Gamma_{\theta,5}. 
\end{equation*}
}
\end{theorem}

\begin{proof}
It is obvious. 
\end{proof}

\subsection{The case where $k\equiv 5 \bmod{10}$  }

\begin{theorem}
\label{thm:k=5-(1/5,1/5)-(1/5,9/5)}
{\it
Suppose that $k\equiv 5 \bmod{10}.$ 
Then, it follows that 
\begin{equation*}
\Ker \nu_{F^k}=
\left\{
M
=
\begin{pmatrix}
a & b \\
c & d
\end{pmatrix}
\in
\Gamma_{\theta,5} \, | \,
M
\equiv 
\pm
\begin{pmatrix}
1 & 0 \\
0 & 1
\end{pmatrix} 
\bmod{5}
\right\}. 
\end{equation*}
As a  coset decomposition of $\Gamma_{\theta,5}$ modulo $\Ker \nu_{F^k},$ 
we may choose 
$
\begin{pmatrix}
1 & 0 \\
0 & 1
\end{pmatrix}
$ 
and 
$
\begin{pmatrix}
0 & -1 \\
1 & 0
\end{pmatrix}. 
$ 
}
\end{theorem}

\begin{proof}
It is obvious. 
\end{proof}

\subsection{The case where $k\equiv \pm1, \pm 3 \bmod{10}$  }

\begin{theorem}
\label{thm:k=pm1-pm3-(1/5,1/5)-(1/5,9/5)}
{\it
Suppose that $k\equiv \pm1, \pm3 \bmod{10}.$ 
Then, it follows that 
\begin{align*}
&\Ker \nu_{F^k} \\
=&
\left\{
M
=
\begin{pmatrix}
a & b \\
c & d
\end{pmatrix}
\in
\Gamma_{\theta,5} \, | \, \,
\frac{b}{5} \equiv \frac{c}{5} \bmod{5} \,\,   (b,c \in 5 \mathbb{Z}) 
\right\} \\
=&
\Big\{
M
\in
\Gamma_{\theta,5} \, | \, \,
M
\equiv 
\pm
\begin{pmatrix}
1 & 0 \\
0 & 1
\end{pmatrix}, 
\pm
\begin{pmatrix}
1 & 5 \\
5 & 1
\end{pmatrix}, 
\pm
\begin{pmatrix}
1 & -5 \\
-5 & 1
\end{pmatrix}, 
\pm
\begin{pmatrix}
1 & 10 \\
10 & 1
\end{pmatrix}, 
\pm
\begin{pmatrix}
1 & -10 \\
-10 & 1
\end{pmatrix},  \\
&
\hspace{25mm}
\pm
\begin{pmatrix}
6 & 0 \\
0 & -4
\end{pmatrix}, 
\pm
\begin{pmatrix}
6 & 5 \\
5& -4
\end{pmatrix}, 
\pm
\begin{pmatrix}
6 & -5 \\
-5& -4
\end{pmatrix}, 
\pm
\begin{pmatrix}
6 & 10 \\
10& -4
\end{pmatrix}, 
\pm
\begin{pmatrix}
6 & -10 \\
-10& -4
\end{pmatrix}, \\
&
\hspace{25mm}
\pm
\begin{pmatrix}
11 & 0 \\
0 & -9
\end{pmatrix}, 
\pm
\begin{pmatrix}
11 & 5 \\
5 & -9
\end{pmatrix}, 
\pm
\begin{pmatrix}
11 & -5 \\
-5 & -9
\end{pmatrix}, 
\pm
\begin{pmatrix}
11 & 10 \\
10 & -9
\end{pmatrix}, 
\pm
\begin{pmatrix}
11 & -10 \\
-10 & -9
\end{pmatrix},  \\
&
\hspace{25mm}
\pm
\begin{pmatrix}
-9 & 0 \\
0 & 11
\end{pmatrix}, 
\pm
\begin{pmatrix}
-9 & 5 \\
5 & 11
\end{pmatrix}, 
\pm
\begin{pmatrix}
-9 & -5 \\
-5 & 11
\end{pmatrix}, 
\pm
\begin{pmatrix}
-9 & 10 \\
10 & 11
\end{pmatrix}, 
\pm
\begin{pmatrix}
-9 & -10 \\
-10 & 11
\end{pmatrix},  \\
&
\hspace{25mm}
\pm
\begin{pmatrix}
-4 & 0 \\
0 & 6
\end{pmatrix}, 
\pm
\begin{pmatrix}
-4 & 5 \\
5 & 6
\end{pmatrix}, 
\pm
\begin{pmatrix}
-4 & -5 \\
-5 & 6
\end{pmatrix}, 
\pm
\begin{pmatrix}
-4 & 10 \\
10 & 6
\end{pmatrix}, 
\pm
\begin{pmatrix}
-4 & -10 \\
-10 & 6
\end{pmatrix}  
\bmod{25}
\Big\}. 
\end{align*}
As a  coset decomposition of $\Gamma_{\theta,5}$ modulo $\Ker \nu_{F^k},$ 
we may choose 
\begin{equation*}
\begin{pmatrix}
1 & n \\
0 & 1
\end{pmatrix} 
\,\, 
(n=0,5,10,15,20). 
\end{equation*}
}
\end{theorem}

\begin{proof}
From the definition of $f(M)$, 
it follows that
\begin{equation*}
M \in 
\Ker \nu_{F^k} 
\Longleftrightarrow 
f(M) \equiv 0 \bmod{10},
\end{equation*}
which implies that 
$
M
\equiv 
\pm
\begin{pmatrix}
1 & 0 \\
0 & 1
\end{pmatrix}
\bmod{5}. 
$ 
\par
Suppose that 
$
M
\equiv 
\begin{pmatrix}
1 & 0 \\
0 & 1
\end{pmatrix}
\bmod{5}, 
$ 
which implies that 
\begin{equation*}
f(M)\equiv 0 \bmod{10}
\Longleftrightarrow 
\frac{b}{5}\equiv \frac{c}{5} \bmod{5} 
\Longleftrightarrow 
(b,c)
\equiv 
(0,0), 
\pm (5,5), 
\pm (10,10) \bmod{25}. 
\end{equation*}
Since $ad-bc=1,$ it follows that $ad\equiv 1 \bmod{25}, $ 
which implies that 
\begin{equation*}
(a,d)\equiv (1,1), (6,-4), (11,-9), (-9,11), (-4,6) \bmod{25}. 
\end{equation*}
\par
Suppose that 
$
M
\equiv 
\begin{pmatrix}
-1 & 0 \\
0 & -1
\end{pmatrix}
\bmod{5}, 
$ 
which implies that 
\begin{equation*}
f(M)\equiv 0 \bmod{10}
\Longleftrightarrow 
\frac{b}{5}\equiv \frac{c}{5} \bmod{5} 
\Longleftrightarrow 
(b,c)
\equiv 
(0,0), 
\pm (5,5), 
\pm (10,10) \bmod{25}. 
\end{equation*}
Since $ad-bc=1,$ it follows that $ad\equiv 1 \bmod{25}, $ 
which implies that 
\begin{equation*}
(a,d)\equiv (-1,-1), (4,-6), (9,-11), (-11,9), (-6,4) \bmod{25}. 
\end{equation*}
\par
The coset decomposition can be obtained by considering the values of $\nu_{F^{k}  }(M).$
\end{proof}

\subsection{The case where $k\equiv \pm2, \pm 4 \bmod{10}$  }

\begin{theorem}
\label{thm:k=pm2-pm4-(1/5,1/5)-(1/5,9/5)}
{\it
Suppose that $k\equiv \pm2, \pm4 \bmod{10}.$ 
Then, it follows that 
\begin{align*}
&
\Ker \nu_{F^k} \\
=&
\left\{
M
=
\begin{pmatrix}
a & b \\
c & d
\end{pmatrix}
\in
\Gamma_{\theta,5} \, | \, \,
\frac{b}{5} \equiv \frac{c}{5} \bmod{5} \,\,   (b,c \in 5 \mathbb{Z}) \,\,
\mathrm{or} \,\,
\frac{a}{5} \equiv -\frac{d}{5} \bmod{5} \,\,   (a,d \in 5 \mathbb{Z}) 
\right\} \\
=&
\Big\{
M
\in
\Gamma_{\theta,5} \, | \, \,
M
\equiv 
\pm
\begin{pmatrix}
1 & 0 \\
0 & 1
\end{pmatrix}, 
\pm
\begin{pmatrix}
1 & 5 \\
5 & 1
\end{pmatrix}, 
\pm
\begin{pmatrix}
1 & -5 \\
-5 & 1
\end{pmatrix}, 
\pm
\begin{pmatrix}
1 & 10 \\
10 & 1
\end{pmatrix}, 
\pm
\begin{pmatrix}
1 & -10 \\
-10 & 1
\end{pmatrix},  \\
&
\hspace{6mm}
\pm
\begin{pmatrix}
6 & 0 \\
0 & -4
\end{pmatrix}, 
\pm
\begin{pmatrix}
6 & 5 \\
5& -4
\end{pmatrix}, 
\pm
\begin{pmatrix}
6 & -5 \\
-5& -4
\end{pmatrix}, 
\pm
\begin{pmatrix}
6 & 10 \\
10& -4
\end{pmatrix}, 
\pm
\begin{pmatrix}
6 & -10 \\
-10& -4
\end{pmatrix}, \\
&
\hspace{6mm}
\pm
\begin{pmatrix}
11 & 0 \\
0 & -9
\end{pmatrix}, 
\pm
\begin{pmatrix}
11 & 5 \\
5 & -9
\end{pmatrix}, 
\pm
\begin{pmatrix}
11 & -5 \\
-5 & -9
\end{pmatrix}, 
\pm
\begin{pmatrix}
11 & 10 \\
10 & -9
\end{pmatrix}, 
\pm
\begin{pmatrix}
11 & -10 \\
-10 & -9
\end{pmatrix},  \\
&
\hspace{6mm}
\pm
\begin{pmatrix}
-9 & 0 \\
0 & 11
\end{pmatrix}, 
\pm
\begin{pmatrix}
-9 & 5 \\
5 & 11
\end{pmatrix}, 
\pm
\begin{pmatrix}
-9 & -5 \\
-5 & 11
\end{pmatrix}, 
\pm
\begin{pmatrix}
-9 & 10 \\
10 & 11
\end{pmatrix}, 
\pm
\begin{pmatrix}
-9 & -10 \\
-10 & 11
\end{pmatrix},  \\
&
\hspace{6mm}
\pm
\begin{pmatrix}
-4 & 0 \\
0 & 6
\end{pmatrix}, 
\pm
\begin{pmatrix}
-4 & 5 \\
5 & 6
\end{pmatrix}, 
\pm
\begin{pmatrix}
-4 & -5 \\
-5 & 6
\end{pmatrix}, 
\pm
\begin{pmatrix}
-4 & 10 \\
10 & 6
\end{pmatrix}, 
\pm
\begin{pmatrix}
-4 & -10 \\
-10 & 6
\end{pmatrix},  \\
&
\hspace{6mm}
\pm
\begin{pmatrix}
0 & -1 \\
1 & 0
\end{pmatrix}, 
\pm
\begin{pmatrix}
0 & 4 \\
6 & 0
\end{pmatrix}, 
\pm
\begin{pmatrix}
0 & 9 \\
11 & 0
\end{pmatrix}, 
\pm
\begin{pmatrix}
0 & -11 \\
-9 & 0
\end{pmatrix}, 
\pm
\begin{pmatrix}
0 & -6 \\
-4 & 0
\end{pmatrix},  \\
&
\hspace{6mm}
\pm
\begin{pmatrix}
5 & -1 \\
1 & -5
\end{pmatrix}, 
\pm
\begin{pmatrix}
5 & 4 \\
6 & -5
\end{pmatrix}, 
\pm
\begin{pmatrix}
5 & 9 \\
11 & -5
\end{pmatrix}, 
\pm
\begin{pmatrix}
5 & -11 \\
-9 & -5
\end{pmatrix}, 
\pm
\begin{pmatrix}
5 & -6 \\
-4 & -5
\end{pmatrix},  \\
&
\hspace{6mm}
\pm
\begin{pmatrix}
-5 & -1 \\
1 & 5
\end{pmatrix}, 
\pm
\begin{pmatrix}
-5 & 4 \\
6 & 5
\end{pmatrix}, 
\pm
\begin{pmatrix}
-5 & 9 \\
11 & 5
\end{pmatrix}, 
\pm
\begin{pmatrix}
-5 & -11 \\
-9 & 5
\end{pmatrix}, 
\pm
\begin{pmatrix}
-5 & -6 \\
-4 & 5
\end{pmatrix},  \\
&
\hspace{6mm}
\pm
\begin{pmatrix}
10 & -1 \\
1 & -10
\end{pmatrix}, 
\pm
\begin{pmatrix}
10 & 4 \\
6 & -10
\end{pmatrix}, 
\pm
\begin{pmatrix}
10 & 9 \\
11 & -10
\end{pmatrix}, 
\pm
\begin{pmatrix}
10 & -11 \\
-9 & -10
\end{pmatrix}, 
\pm
\begin{pmatrix}
10 & -6 \\
-4 & -10
\end{pmatrix},  \\
&
\hspace{6mm}
\pm
\begin{pmatrix}
-10 & -1 \\
1 & 10
\end{pmatrix}, 
\pm
\begin{pmatrix}
-10 & 4 \\
6 & 10
\end{pmatrix}, 
\pm
\begin{pmatrix}
-10 & 9 \\
11 & 10
\end{pmatrix}, 
\pm
\begin{pmatrix}
-10 & -11 \\
-9 & 10
\end{pmatrix}, 
\pm
\begin{pmatrix}
-10 & -6 \\
-4 & 10
\end{pmatrix}
%
%
\bmod{25}
\Big\}. 
\end{align*}
As a  coset decomposition of $\Gamma_{\theta,5}$ modulo $\Ker \nu_{F^k},$ 
we may choose 
\begin{equation*}
\begin{pmatrix}
1 & n \\
0 & 1
\end{pmatrix}, \,\,
\begin{pmatrix}
n & -1 \\
1 & 0
\end{pmatrix}
\,\, 
(n=0,5,10,15,20). 
\end{equation*}
}
\end{theorem}

\begin{proof}
Suppose that 
$
M
\equiv 
\pm
\begin{pmatrix}
1 & 0 \\
0 & 1
\end{pmatrix}
\bmod{5}. 
$ 
From the definition of $f(M),$ 
it follows that 
\begin{equation*}
M \in 
\Ker \nu_{F^k} 
\Longleftrightarrow 
\frac{b}{5} \equiv \frac{c}{5}  \bmod{5}. 
\end{equation*}
\par
Suppose that 
$
M
\equiv 
\pm
\begin{pmatrix}
0 & -1 \\
1 & 0
\end{pmatrix}
\bmod{5}. 
$ 
From the definition of $f(M),$ 
it follows that 
\begin{equation*}
M \in 
\Ker \nu_{F^k} 
\Longleftrightarrow 
\frac{a}{5} \equiv -\frac{d}{5}  \bmod{5} 
\Longleftrightarrow
(a,d)
\equiv 
(0,0), 
\pm
(5,-5),  
\pm(10,-10) \bmod{25}. 
\end{equation*}
Since $ad-bc=1,$ it follows that $bc\equiv -1 \bmod{25},$ 
which implies that 
\begin{equation*}
(b,c)\equiv 
\pm(-1,1), 
\pm(4,6), 
\pm(9,11), 
\pm(-11,-9), 
\pm(-6,-4) \bmod{25}. 
\end{equation*}
\par
The coset decomposition can be obtained by considering the values of $\nu_{F^{k}  }(M).$
\end{proof}

\section{Modular forms on $\Gamma_{\theta,5}$ (2) }
\label{sec:modularforms-(3/5,3/5),(3/5,7/5)}

\subsection{The case where 
$
M
\equiv
\begin{pmatrix}
1 & 0 \\
0 & 1
\end{pmatrix} 
\bmod{5}. 
$ 
}

\begin{theorem}
\label{thm:(3/5,3/5),(3/5,7/5)-I}
{\it
Set 
$
M
=
\begin{pmatrix}
a & b \\
c & d
\end{pmatrix}
\in
\Gamma_{\theta,5}
$ 
with 
$
M
\equiv
\begin{pmatrix}
1 & 0 \\
0 & 1
\end{pmatrix} 
\bmod{5}. 
$ 
Then the multiplier system of 
$
\theta 
\left[
\begin{array}{c}
\frac35 \vspace{1mm} \\
\frac35
\end{array}
\right] \left(0, \tau  \right) 
\theta 
\left[
\begin{array}{c}
\frac35 \vspace{1mm} \\
\frac75
\end{array}
\right] \left(0, \tau  \right) 
$ 
is given by 
\begin{equation*}
\nu(M)
=
\nu_{\eta}^6(M)
\exp
\left[
\frac{\pi i}{5}
\left(
-
\frac{6b}{5}
-
\frac{2ab}{5}
-
\frac{2cd}{5}
\right)
\right]. 
\end{equation*}
}
\end{theorem}

\begin{proof}
Since $ad-bc=1,$ 
it follows that 
\begin{equation*}
(a,c)\equiv (1,1),(1,0),(0,1) \bmod{2} \, \mathrm{and} \,
(b,d)\equiv (1,1),(1,0),(0,1) \bmod{2},
\end{equation*}
which implies that 
\begin{align*}
&
\frac{3a}{5}+\frac{3c}{5}-ac-\frac35 =
\frac{3(a-1)}{5}+\frac{3c}{5}-ac \equiv 0 \bmod{2}, \\
&
\frac{3b}{5}+\frac{3d}{5}-bd-\frac35 =
\frac{3b}{5}+\frac{3(d-1)}{5}-bd \equiv 0 \bmod{2},  \\
&
\frac{3a}{5}+\frac{7c}{5}-ac-\frac35 =
\frac{3(a-1)}{5}+\frac{7c}{5}-ac \equiv 0 \bmod{2}, \\
&
\frac{3b}{5}+\frac{7d}{5}-bd-\frac35 =
\frac{3b}{5}+\frac{7(d-1)}{5}-bd \equiv 0 \bmod{2}. 
\end{align*} 
By equation (\ref{eqn:char-even}) and 
Theorem \ref{thm:transformation-general}, 
we find that 
\begin{align*}
&
\theta 
\left[
\begin{array}{c}
\frac35 \vspace{1mm} \\
\frac35
\end{array}
\right] \left(0, M\tau  \right) 
\theta 
\left[
\begin{array}{c}
\frac35 \vspace{1mm} \\
\frac75
\end{array}
\right] \left(0, M\tau  \right) \\
=&
\nu_{\eta}^6(M)
\cdot
(c\tau+d) \cdot
\theta 
\left[
\begin{array}{c}
3a/5+3c/5-ac \vspace{1mm} \\
3b/5+3d/5-bd
\end{array}
\right] \left(0, \tau  \right) 
\theta 
\left[
\begin{array}{c}
3a/5+7c/5-ac \vspace{1mm} \\
3b/5+7d/5-bd
\end{array}
\right] \left(0, \tau  \right) \\
&
\times
\exp
(\pi i)
\left[
\frac{2a}{5}
(b+d-bd)
-
\frac{2}{25}
ab
-
\frac{2}{25}
cd
-\frac25
\right]  \\
=&
\nu_{\eta}^6(M)
\cdot
(c\tau+d) \cdot
\theta 
\left[
\begin{array}{c}
\frac35 \vspace{1mm} \\
\frac35
\end{array}
\right] \left(0, \tau  \right) 
\theta 
\left[
\begin{array}{c}
\frac35 \vspace{1mm} \\
\frac75
\end{array}
\right] \left(0, \tau  \right) \\
&
\times
\exp(\pi i)
\left[
\frac35 \cdot \frac12
\left(
\frac{3b}{5}
+\frac{3(d-1)}{5}
-bd
\right)
\right] 
\cdot
\exp(\pi i)
\left[
\frac35 \cdot \frac12
\left(
\frac{3b}{5}
+\frac{7(d-1)}{5}
-bd
\right)
\right]   \\
&
\times
\exp
(\pi i)
\left[
\frac{2a}{5}
(b+d-bd)
-
\frac{2}{25}
ab
-
\frac{2}{25}
cd
-\frac25
\right]  \\
=&
\nu_{\eta}^6(M)
\cdot
(c\tau+d) \cdot
\theta 
\left[
\begin{array}{c}
\frac35 \vspace{1mm} \\
\frac35
\end{array}
\right] \left(0, \tau  \right) 
\theta 
\left[
\begin{array}{c}
\frac35 \vspace{1mm} \\
\frac75
\end{array}
\right] \left(0, \tau  \right)
\cdot
\exp(\pi i) 
\left[
E
\right],
\end{align*}
where 
\begin{align*}
E=&
\frac35 \cdot \frac12
\left(
\frac{3b}{5}
+\frac{3(d-1)}{5}
-bd
\right) 
+
\frac35 \cdot \frac12
\left(
\frac{3b}{5}
+\frac{7(d-1)}{5}
-bd
\right)
+
\frac{2a}{5}
(b+d-bd)  \\
&\hspace{100mm}
-
\frac{2}{25}
ab
-
\frac{2}{25}
cd
-\frac25   \\
=&
\frac{1}{25}
(
9b+15d-15-15bd
)
+
\frac{2a}{5}
(b+d-bd)
-
\frac{2}{25}ab
-
\frac{2}{25}
cd
-
\frac25  \\
=&
\frac{2}{25}
\left\{
- 
\frac{15}{2} (b-1)(d-1)-3b
+5(a-1)
-
5a(b-1)(d-1)
-ab
-cd
\right\} \\
\equiv&
\frac{2}{25}
(
-3b-ab-cd
)
\bmod{2}   \\
=&
\frac15
\left(
-\frac{6b}{5}
-
\frac{2ab}{5}
-
\frac{2cd}{5}
\right), 
\end{align*}
which proves the theorem. 
\end{proof}

\subsection{The case where 
$
M
\equiv
\begin{pmatrix}
-1 & 0 \\
0 & -1
\end{pmatrix} 
\bmod{5}. 
$ 
}

\begin{theorem}
\label{thm:(3/5,3/5),(3/5,7/5)-minusI}
{\it
Set 
$
M
=
\begin{pmatrix}
a & b \\
c & d
\end{pmatrix}
\in
\Gamma_{\theta,5}
$ 
with 
$
M
\equiv
\begin{pmatrix}
-1 & 0 \\
0 & -1
\end{pmatrix} 
\bmod{5}. 
$ 
Then the multiplier system of 
$
\theta 
\left[
\begin{array}{c}
\frac35 \vspace{1mm} \\
\frac35
\end{array}
\right] \left(0, \tau  \right) 
\theta 
\left[
\begin{array}{c}
\frac35 \vspace{1mm} \\
\frac75
\end{array}
\right] \left(0, \tau  \right) 
$ 
is given by 
\begin{equation*}
\nu(M)
=
\nu_{\eta}^6(M)
\exp
\left[
\frac{\pi i}{5}
\left(
-
\frac{24b}{5}
-
\frac{2ab}{5}
-
\frac{2cd}{5}
\right)
\right]. 
\end{equation*}
}
\end{theorem}

\begin{proof}
Since $ad-bc=1,$ 
it follows that 
\begin{equation*}
(a,c)\equiv (1,1),(1,0),(0,1) \bmod{2} \, \mathrm{and} \,
(b,d)\equiv (1,1),(1,0),(0,1) \bmod{2},
\end{equation*}
which implies that 
\begin{align*}
&
\frac{3a}{5}+\frac{3c}{5}-ac-\left(-\frac35\right) =
\frac{3(a+1)}{5}+\frac{3c}{5}-ac \equiv 0 \bmod{2}, \\
&
\frac{3b}{5}+\frac{3d}{5}-bd-\left(-\frac35 \right) =
\frac{3b}{5}+\frac{3(d+1)}{5}-bd  \equiv 0 \bmod{2},  \\
&
\frac{3a}{5}+\frac{7c}{5}-ac-\left(-\frac35 \right) =
\frac{3(a+1)}{5}+\frac{7c}{5}-ac \equiv 0 \bmod{2}, \\
&
\frac{3b}{5}+\frac{7d}{5}-bd-\left(-\frac75 \right) =
\frac{3b}{5}+\frac{7(d+1)}{5}-bd  \equiv 0 \bmod{2}. 
\end{align*} 
By equations (\ref{eqn:char-even}), (\ref{eqn:character-minus}) and 
Theorem \ref{thm:transformation-general}, 
we find that 
\begin{align*}
&
\theta 
\left[
\begin{array}{c}
\frac35 \vspace{1mm} \\
\frac35
\end{array}
\right] \left(0, M\tau  \right) 
\theta 
\left[
\begin{array}{c}
\frac35 \vspace{1mm} \\
\frac75
\end{array}
\right] \left(0, M\tau  \right) \\
=&
\nu_{\eta}^6(M)
\cdot
(c\tau+d) \cdot
\theta 
\left[
\begin{array}{c}
3a/5+3c/5-ac \vspace{1mm} \\
3b/5+3d/5-bd
\end{array}
\right] \left(0, \tau  \right) 
\theta 
\left[
\begin{array}{c}
3a/5+7c/5-ac \vspace{1mm} \\
3b/5+7d/5-bd
\end{array}
\right] \left(0, \tau  \right) \\
&
\times
\exp
(\pi i)
\left[
\frac{2a}{5}
(b+d-bd)
-
\frac{2}{25}
ab
-
\frac{2}{25}
cd
-\frac25
\right]  \\
=&
\nu_{\eta}^6(M)
\cdot
(c\tau+d) \cdot
\theta 
\left[
\begin{array}{c}
\frac35 \vspace{1mm} \\
\frac35
\end{array}
\right] \left(0, \tau  \right) 
\theta 
\left[
\begin{array}{c}
\frac35 \vspace{1mm} \\
\frac75
\end{array}
\right] \left(0, \tau  \right) \\
&
\times
\exp(\pi i)
\left[
-
\frac35 \cdot \frac12
\left(
\frac{3b}{5}
+\frac{3(d+1)}{5}
-bd
\right)
\right] 
\cdot
\exp(\pi i)
\left[
-
\frac35 \cdot \frac12
\left(
\frac{3b}{5}
+\frac{7(d+1)}{5}
-bd
\right)
\right]   \\
&
\times
\exp
(\pi i)
\left[
\frac{2a}{5}
(b+d-bd)
-
\frac{2}{25}
ab
-
\frac{2}{25}
cd
-\frac25
\right]  \\
=&
\nu_{\eta}^6(M)
\cdot
(c\tau+d) \cdot
\theta 
\left[
\begin{array}{c}
\frac35 \vspace{1mm} \\
\frac35
\end{array}
\right] \left(0, \tau  \right) 
\theta 
\left[
\begin{array}{c}
\frac35 \vspace{1mm} \\
\frac75
\end{array}
\right] \left(0, \tau  \right)
\cdot
\exp(\pi i) 
\left[
E
\right],
\end{align*}
where 
\begin{align*}
E=&
-
\frac35 \cdot \frac12
\left(
\frac{3b}{5}
+\frac{3(d+1)}{5}
-bd
\right) 
-
\frac35 \cdot \frac12
\left(
\frac{3b}{5}
+\frac{7(d+1)}{5}
-bd
\right)
+
\frac{2a}{5}
(b+d-bd) \\
&\hspace{100mm}
-
\frac{2}{25}
ab
-
\frac{2}{25}
cd
-\frac25   \\
=&
\frac{1}{25}
(
-9b-15d-15+15bd
)
+
\frac{2a}{5}
(b+d-bd)
-
\frac{2}{25}ab
-
\frac{2}{25}
cd
-
\frac25  \\
=&
\frac{2}{25}
\left\{
\frac{15}{2} (b-1)(d+1)-12b
-5(a+1)
-
5a(b-1)(d+1)
+10ab
-ab
-cd
\right\} \\
\equiv&
\frac{2}{25}
(
-12b-ab-cd
)
\bmod{2}   \\
=&
\frac15
\left(
-\frac{24b}{5}
-
\frac{2ab}{5}
-
\frac{2cd}{5}
\right), 
\end{align*}
which proves the theorem. 
\end{proof}

\subsection{The case where 
$
M
\equiv
\begin{pmatrix}
0 & -1 \\
1 & 0
\end{pmatrix} 
\bmod{5}. 
$ 
}

\begin{theorem}
\label{thm:(3/5,3/5),(3/5,7/5)-T}
{\it
Set 
$
M
=
\begin{pmatrix}
a & b \\
c & d
\end{pmatrix}
\in
\Gamma_{\theta,5}
$ 
with 
$
M
\equiv
\begin{pmatrix}
0 & -1 \\
1 & 0
\end{pmatrix} 
\bmod{5}. 
$ 
Then the multiplier system of 
$
\theta 
\left[
\begin{array}{c}
\frac35 \vspace{1mm} \\
\frac35
\end{array}
\right] \left(0, \tau  \right) 
\theta 
\left[
\begin{array}{c}
\frac35 \vspace{1mm} \\
\frac75
\end{array}
\right] \left(0, \tau  \right) 
$ 
is given by 
\begin{equation*}
\nu(M)
=
\nu_{\eta}^6(M)
\exp
\left[
\frac{\pi i}{5}
\left(
-
5
-
\frac{6d}{5}
-
\frac{2ab}{5}
-
\frac{2cd}{5}
\right)
\right]. 
\end{equation*}
}
\end{theorem}

\begin{proof}
Since $ad-bc=1,$ 
it follows that 
\begin{equation*}
(a,c)\equiv (1,1),(1,0),(0,1) \bmod{2} \, \mathrm{and} \,
(b,d)\equiv (1,1),(1,0),(0,1) \bmod{2},
\end{equation*}
which implies that 
\begin{align*}
&
\frac{3a}{5}+\frac{3c}{5}-ac-\frac35 =
\frac{3a}{5}+\frac{3(c-1)}{5}-ac \equiv 0 \bmod{2}, \\
&
\frac{3b}{5}+\frac{3d}{5}-bd-\frac75 =
\frac{3b-7}{5}+\frac{3d}{5}-bd \equiv 0 \bmod{2},  \\
&
\frac{3a}{5}+\frac{7c}{5}-ac-\left(-\frac35 \right) =
\frac{3a}{5}+\frac{7c+3}{5}-ac \equiv 0 \bmod{2}, \\
&
\frac{3b}{5}+\frac{7d}{5}-bd-\left(-\frac35 \right)=
\frac{3(b+1)}{5}+\frac{7d}{5}-bd 
\equiv 0 \bmod{2}. 
\end{align*} 
By equations (\ref{eqn:char-even}), (\ref{eqn:character-minus}) and 
Theorem \ref{thm:transformation-general}, 
we find that 
\begin{align*}
&
\theta 
\left[
\begin{array}{c}
\frac35 \vspace{1mm} \\
\frac35
\end{array}
\right] \left(0, M\tau  \right) 
\theta 
\left[
\begin{array}{c}
\frac35 \vspace{1mm} \\
\frac75
\end{array}
\right] \left(0, M\tau  \right) \\
=&
\nu_{\eta}^6(M)
\cdot
(c\tau+d) \cdot
\theta 
\left[
\begin{array}{c}
3a/5+3c/5-ac \vspace{1mm} \\
3b/5+3d/5-bd
\end{array}
\right] \left(0, \tau  \right) 
\theta 
\left[
\begin{array}{c}
3a/5+7c/5-ac \vspace{1mm} \\
3b/5+7d/5-bd
\end{array}
\right] \left(0, \tau  \right) \\
&
\times
\exp
(\pi i)
\left[
\frac{2a}{5}
(b+d-bd)
-
\frac{2}{25}
ab
-
\frac{2}{25}
cd
-\frac25
\right]  \\
=&
\nu_{\eta}^6(M)
\cdot
(c\tau+d) \cdot
\theta 
\left[
\begin{array}{c}
\frac35 \vspace{1mm} \\
\frac35
\end{array}
\right] \left(0, \tau  \right) 
\theta 
\left[
\begin{array}{c}
\frac35 \vspace{1mm} \\
\frac75
\end{array}
\right] \left(0, \tau  \right) \\
&
\times
\exp(\pi i)
\left[
\frac35 \cdot \frac12
\left(
\frac{3b-7}{5}
+\frac{3d}{5}
-bd
\right)
\right] 
\cdot
\exp(\pi i)
\left[
-
\frac35 \cdot \frac12
\left(
\frac{3(b+1)}{5}
+\frac{7d}{5}
-bd
\right)
\right]   \\
&
\times
\exp
(\pi i)
\left[
\frac{2a}{5}
(b+d-bd)
-
\frac{2}{25}
ab
-
\frac{2}{25}
cd
-\frac25
\right]  \\
=&
\nu_{\eta}^6(M)
\cdot
(c\tau+d) \cdot
\theta 
\left[
\begin{array}{c}
\frac35 \vspace{1mm} \\
\frac35
\end{array}
\right] \left(0, \tau  \right) 
\theta 
\left[
\begin{array}{c}
\frac35 \vspace{1mm} \\
\frac75
\end{array}
\right] \left(0, \tau  \right)
\cdot
\exp(\pi i) 
\left[
E
\right],
\end{align*}
where 
\begin{align*}
E=&
\frac35 \cdot \frac12
\left(
\frac{3b-7}{5}
+\frac{3d}{5}
-bd
\right) 
-
\frac35 \cdot \frac12
\left(
\frac{3(b+1)}{5}
+\frac{7d}{5}
-bd
\right)
+
\frac{2a}{5}
(b+d-bd)  \\
&\hspace{100mm}
-
\frac{2}{25}
ab
-
\frac{2}{25}
cd
-\frac25   \\
\equiv&
-
1
-
\frac{6d}{25}
-
\frac{2ab}{25}
-
\frac{2cd}{25} \bmod{2}  \\
=&
\frac15
\left[
-5
-
\frac{6d}{5}
-
\frac{2ab}{5}
-
\frac{2cd}{5}
\right], 
\end{align*}
which proves the theorem. 
\end{proof}

\subsection{The case where 
$
M
\equiv
\begin{pmatrix}
0 & 1 \\
-1 & 0
\end{pmatrix} 
\bmod{5}. 
$ 
}

\begin{theorem}
\label{thm:(3/5,3/5),(3/5,7/5)-minusT}
{\it
Set 
$
M
=
\begin{pmatrix}
a & b \\
c & d
\end{pmatrix}
\in
\Gamma_{\theta,5}
$ 
with 
$
M
\equiv
\begin{pmatrix}
0 & 1 \\
-1 & 0
\end{pmatrix} 
\bmod{5}. 
$ 
Then the multiplier system of 
$
\theta 
\left[
\begin{array}{c}
\frac35 \vspace{1mm} \\
\frac35
\end{array}
\right] \left(0, \tau  \right) 
\theta 
\left[
\begin{array}{c}
\frac35 \vspace{1mm} \\
\frac75
\end{array}
\right] \left(0, \tau  \right) 
$ 
is given by 
\begin{equation*}
\nu(M)
=
\nu_{\eta}^6(M)
\exp
\left[
\frac{\pi i}{5}
\left(
-
5
+
\frac{6d}{5}
-
\frac{2ab}{5}
-
\frac{2cd}{5}
\right)
\right]. 
\end{equation*}
}
\end{theorem}

\begin{proof}
Since $ad-bc=1,$ 
it follows that 
\begin{equation*}
(a,c)\equiv (1,1),(1,0),(0,1) \bmod{2} \, \mathrm{and} \,
(b,d)\equiv (1,1),(1,0),(0,1) \bmod{2},
\end{equation*}
which implies that 
\begin{align*}
&
\frac{3a}{5}+\frac{3c}{5}-ac-\left(-\frac35\right) =
\frac{3a}{5}+\frac{3(c+1)}{5}-ac \equiv 0 \bmod{2}, \\
&
\frac{3b}{5}+\frac{3d}{5}-bd-\left(-\frac75 \right) =
\frac{3b+7}{5}+\frac{3d}{5}-bd \equiv 0 \bmod{2},  \\
&
\frac{3a}{5}+\frac{7c}{5}-ac-\frac35  =
\frac{3a}{5}+\frac{7c-3}{5}-ac \equiv 0 \bmod{2}, \\
&
\frac{3b}{5}+\frac{7d}{5}-bd-\frac35 =
\frac{3(b-1)}{5}+\frac{7d}{5}-bd  
\equiv 0 \bmod{2}. 
\end{align*} 
By equations (\ref{eqn:char-even}), (\ref{eqn:character-minus}) and 
Theorem \ref{thm:transformation-general}, 
we find that 
\begin{align*}
&
\theta 
\left[
\begin{array}{c}
\frac35 \vspace{1mm} \\
\frac35
\end{array}
\right] \left(0, M\tau  \right) 
\theta 
\left[
\begin{array}{c}
\frac35 \vspace{1mm} \\
\frac75
\end{array}
\right] \left(0, M\tau  \right) \\
=&
\nu_{\eta}^6(M)
\cdot
(c\tau+d) \cdot
\theta 
\left[
\begin{array}{c}
3a/5+3c/5-ac \vspace{1mm} \\
3b/5+3d/5-bd
\end{array}
\right] \left(0, \tau  \right) 
\theta 
\left[
\begin{array}{c}
3a/5+7c/5-ac \vspace{1mm} \\
3b/5+7d/5-bd
\end{array}
\right] \left(0, \tau  \right) \\
&
\times
\exp
(\pi i)
\left[
\frac{2a}{5}
(b+d-bd)
-
\frac{2}{25}
ab
-
\frac{2}{25}
cd
-\frac25
\right]  \\
=&
\nu_{\eta}^6(M)
\cdot
(c\tau+d) \cdot
\theta 
\left[
\begin{array}{c}
\frac35 \vspace{1mm} \\
\frac35
\end{array}
\right] \left(0, \tau  \right) 
\theta 
\left[
\begin{array}{c}
\frac35 \vspace{1mm} \\
\frac75
\end{array}
\right] \left(0, \tau  \right) \\
&
\times
\exp(\pi i)
\left[
-
\frac35 \cdot \frac12
\left(
\frac{3b+7}{5}
+\frac{3d}{5}
-bd
\right)
\right] 
\cdot
\exp(\pi i)
\left[
\frac35 \cdot \frac12
\left(
\frac{3(b-1)}{5}
+\frac{7d}{5}
-bd
\right)
\right]   \\
&
\times
\exp
(\pi i)
\left[
\frac{2a}{5}
(b+d-bd)
-
\frac{2}{25}
ab
-
\frac{2}{25}
cd
-\frac25
\right]  \\
=&
\nu_{\eta}^6(M)
\cdot
(c\tau+d) \cdot
\theta 
\left[
\begin{array}{c}
\frac35 \vspace{1mm} \\
\frac35
\end{array}
\right] \left(0, \tau  \right) 
\theta 
\left[
\begin{array}{c}
\frac35 \vspace{1mm} \\
\frac75
\end{array}
\right] \left(0, \tau  \right)
\cdot
\exp(\pi i) 
\left[
E
\right],
\end{align*}
where 
\begin{align*}
E=&
-
\frac35 \cdot \frac12
\left(
\frac{3b+7}{5}
+\frac{3d}{5}
-bd
\right) 
+
\frac35 \cdot \frac12
\left(
\frac{3(b-1)}{5}
+\frac{7d}{5}
-bd
\right)
+
\frac{2a}{5}
(b+d-bd)  \\
&\hspace{100mm}
-
\frac{2}{25}
ab
-
\frac{2}{25}
cd
-\frac25   \\
\equiv&
-
1
+
\frac{6d}{25}
-
\frac{2ab}{25}
-
\frac{2cd}{25} \bmod{2}  \\
=&
\frac15
\left[
-5
+
\frac{6d}{5}
-
\frac{2ab}{5}
-
\frac{2cd}{5}
\right], 
\end{align*}
which proves the theorem. 
\end{proof}

\subsection{Summary}

\begin{theorem}
\label{thm:(3/5,3/5),(3/5,7/5)-summary}
{\it
For each 
$
M
=
\begin{pmatrix}
a & b \\
c & d
\end{pmatrix}
\in
\Gamma_{\theta,5}, 
$ 
the multiplier system of 
$$
\theta 
\left[
\begin{array}{c}
\frac35 \vspace{1mm} \\
\frac35
\end{array}
\right] \left(0, \tau  \right) 
\theta 
\left[
\begin{array}{c}
\frac35 \vspace{1mm} \\
\frac75
\end{array}
\right] \left(0, \tau  \right) 
$$ 
is given by 
\begin{equation*}
\nu(M)
=
\nu_{\eta}^6(M)\cdot
\exp\left( \frac{\pi i}{5} E \right), 
\end{equation*}
where $E$ is defined by 
\begin{equation*}
E=
\begin{cases}
\displaystyle
-\frac{6b}{5}-\frac{2ab}{5}-\frac{2cd}{5} 
&
\text{
if 
$
M
\equiv 
\begin{pmatrix}
1 & 0 \\
0 & 1
\end{pmatrix}
\bmod{5}
$
} \\ 
\displaystyle
-\frac{24b}{5}-\frac{2ab}{5}-\frac{2cd}{5} 
&
\text{
if 
$
M
\equiv 
\begin{pmatrix}
-1 & 0 \\
0 & -1
\end{pmatrix}
\bmod{5}
$
} \\ 
\displaystyle
-5
-\frac{6d}{5}-\frac{2ab}{5}-\frac{2cd}{5} 
&
\text{
if 
$
M
\equiv 
\begin{pmatrix}
0 & -1 \\
1 & 0
\end{pmatrix}
\bmod{5}
$
} \\ 
\displaystyle
-5
+\frac{6d}{5}-\frac{2ab}{5}-\frac{2cd}{5} 
&
\text{
if 
$
M
\equiv 
\begin{pmatrix}
0 & 1 \\
-1 & 0
\end{pmatrix}
\bmod{5}. 
$
} 
\end{cases}
\end{equation*}
}
\end{theorem}

\begin{proof}
The theorem follows from 
Theorems 
\ref{thm:(3/5,3/5),(3/5,7/5)-I},  
\ref{thm:(3/5,3/5),(3/5,7/5)-minusI},  
\ref{thm:(3/5,3/5),(3/5,7/5)-T} and 
\ref{thm:(3/5,3/5),(3/5,7/5)-minusT}. 
\end{proof}

\begin{theorem}
\label{thm:(3/5,3/5),(3/5,7/5)-weight-2}
{\it
Set 
$
M
=
\begin{pmatrix}
a & b \\
c & d
\end{pmatrix}
\in
\Gamma_{\theta,5}
$ 
and 
\begin{equation*}
G(\tau)
=
\frac
{
\eta^6(\tau)
}
{
\theta 
\left[
\begin{array}{c}
\frac35 \vspace{1mm} \\
\frac35
\end{array}
\right] \left(0, \tau  \right) 
\theta 
\left[
\begin{array}{c}
\frac35 \vspace{1mm} \\
\frac75
\end{array}
\right] \left(0, \tau  \right) 
}.
\end{equation*}
Then, 
the multiplier system of $G$ is given by 
\begin{equation*}
\nu_{G}
=
\exp
\left[
\frac{\pi i}{5}
g(M)
\right], 
\end{equation*}
where $g(M)$ is defined by 
\begin{equation*}
g(M)=
\begin{cases}
\displaystyle
\frac{6b}{5}+\frac{2ab}{5}+\frac{2cd}{5} 
&
\text{
if 
$
M
\equiv 
\begin{pmatrix}
1 & 0 \\
0 & 1
\end{pmatrix}
\bmod{5}
$
} \\ 
\displaystyle
\frac{24b}{5}+\frac{2ab}{5}+\frac{2cd}{5} 
&
\text{
if 
$
M
\equiv 
\begin{pmatrix}
-1 & 0 \\
0 & -1
\end{pmatrix}
\bmod{5}
$
} \\ 
\displaystyle
5
+\frac{6d}{5}+\frac{2ab}{5}+\frac{2cd}{5} 
&
\text{
if 
$
M
\equiv 
\begin{pmatrix}
0 & -1 \\
1 & 0
\end{pmatrix}
\bmod{5}
$
} \\ 
\displaystyle
5
-\frac{6d}{5}+\frac{2ab}{5}+\frac{2cd}{5} 
&
\text{
if 
$
M
\equiv 
\begin{pmatrix}
0 & 1 \\
-1 & 0
\end{pmatrix}
\bmod{5}. 
$
} 
\end{cases}
\end{equation*}
}
\end{theorem}
\noindent
For each $k\in \mathbb{Z},$ 
we have 
\begin{equation*}
\nu_{G^k}(M)
=
\exp
\left[
\frac{\pi i k}{5}
g(M)
\right]. 
\end{equation*}

\subsection{The case where $k\equiv 0 \bmod{10}$  }

\begin{theorem}
\label{thm:k=0-(3/5,3/5)-(3/5,7/5)}
{\it
Suppose that $k\equiv 0 \bmod{10}.$ 
Then, it follows that 
\begin{equation*}
\Ker \nu_{G^k}=\Gamma_{\theta,5}. 
\end{equation*}
}
\end{theorem}

\begin{proof}
It is obvious. 
\end{proof}

\subsection{The case where $k\equiv 5 \bmod{10}$  }

\begin{theorem}
\label{thm:k=5-(3/5,3/5)-(3/5,7/5)}
{\it
Suppose that $k\equiv 5 \bmod{10}.$ 
Then, it follows that 
\begin{equation*}
\Ker \nu_{G^k}=
\left\{
M
=
\begin{pmatrix}
a & b \\
c & d
\end{pmatrix}
\in
\Gamma_{\theta,5} \, | \,
M
\equiv 
\pm
\begin{pmatrix}
1 & 0 \\
0 & 1
\end{pmatrix} 
\bmod{5}
\right\}. 
\end{equation*}
}
\end{theorem}

\begin{proof}
It is obvious. 
\end{proof}

\subsection{The case where $k\equiv \pm1, \pm 3 \bmod{10}$  }

\begin{theorem}
\label{thm:k=pm1-pm3-(3/5,3/5)-(3/5,7/5)}
{\it
Suppose that $k\equiv \pm1, \pm3 \bmod{10}.$ 
Then, it follows that 
\begin{align*}
\Ker \nu_{G^k}=&
\left\{
M
=
\begin{pmatrix}
a & b \\
c & d
\end{pmatrix}
\in
\Gamma_{\theta,5} \, | \, \,
\frac{b}{5} \equiv \frac{c}{5} \bmod{5} \,\,   (b,c \in 5 \mathbb{Z}) 
\right\}  \\
=&
\Ker \nu_{F^k}.
\end{align*}
}
\end{theorem}

\begin{proof}
The theorem can be proved in the same way as Theorem \ref{thm:k=pm1-pm3-(1/5,1/5)-(1/5,9/5)}. 
\end{proof}

\subsection{The case where $k\equiv \pm2, \pm 4 \bmod{10}$  }

\begin{theorem}
\label{thm:k=pm2-pm4-(3/5,3/5)-(3/5,7/5)}
{\it
Suppose that $k\equiv \pm2, \pm4 \bmod{10}.$ 
Then, it follows that 
\begin{align*}
\Ker \nu_{G^k}=&
\left\{
M
=
\begin{pmatrix}
a & b \\
c & d
\end{pmatrix}
\in
\Gamma_{\theta,5} \, | \, \,
\frac{b}{5} \equiv \frac{c}{5} \bmod{5} \,\,   (b,c \in 5 \mathbb{Z}) \,\,
\mathrm{or} \,\,
\frac{a}{5} \equiv -\frac{d}{5} \bmod{5} \,\,   (a,d \in 5 \mathbb{Z}) 
\right\}  \\
=&\Ker \nu_{F^k}. 
\end{align*}
}
\end{theorem}

\begin{proof}
The theorem can be proved in the same way as Theorem \ref{thm:k=pm2-pm4-(1/5,1/5)-(1/5,9/5)}. 
\end{proof}

\section{Modular form on $\Gamma_{\theta,5}$ (3)}
\label{sec:level5-eta-product}

\begin{proposition}
\label{prop:character-level5}
{\it
For every $\tau\in\mathscr{H},$
set 
\begin{equation*}
H(\tau)= \eta \left(\frac{\tau-1}{5} \right) \eta\left(\frac{\tau+1}{5} \right).
\end{equation*}
The multiplier system $\nu_{H}$ is a character. 
}
\end{proposition}

\begin{proof}
Set 
$
M=
\begin{pmatrix}
a & b \\
c & d
\end{pmatrix}
\in 
\Gamma_{\theta,5}. 
$ 
We note that $\nu_{H}$ is a character if the weight $r$ is a rational integer.  
\par
We first suppose that 
$
M
\equiv \pm
\begin{pmatrix}
1 & 0 \\
0 & 1
\end{pmatrix} 
\bmod{5},$ 
which implies that  
\begin{align*}
H(M\tau)=&
\eta \left(\frac{ M \tau-1}{5} \right) \eta\left(\frac{ M\tau+1}{5} \right)
=
\eta \left( M_1 \left( \frac{  \tau-1}{5} \right)  \right) \eta\left( M_2  \left( \frac{ \tau+1}{5} \right)  \right) \\
=&
\nu_{\eta}(M_1) 
\left(
5c \cdot 
\frac{\tau-1}{5}
+(d+c)
\right)^{\frac12}
\eta \left(\frac{\tau-1}{5} \right)
\nu_{\eta} (M_2)
\left(
5c \cdot 
\frac{\tau+1}{5}
+(d-c)
\right)^{\frac12}
\eta\left(\frac{\tau+1}{5} \right)  \\
=&
\nu_{\eta}(M_1) \nu_{\eta} (M_2)
(c\tau+d)H(\tau),
\end{align*}
where 
\begin{equation*}
M_1
=
\begin{pmatrix}
a-c & \frac{ b-d+a-c }{5} \\
5c & d+c
\end{pmatrix}, 
M_2
=
\begin{pmatrix}
a+c & \frac{ b+d-a-c }{5} \\
5c & d-c
\end{pmatrix}. 
\end{equation*}
\par
We next assume that 
$
M
\equiv \pm
\begin{pmatrix}
0 & -1 \\
1 & 0
\end{pmatrix} 
\bmod{5},$ 
which implies that  
\begin{align*}
H(M\tau)=&
\eta \left(\frac{ M \tau-1}{5} \right) \eta\left(\frac{ M\tau+1}{5} \right)
=
\eta \left( M_1 \left( \frac{  \tau+1}{5} \right)  \right) \eta\left( M_2  \left( \frac{ \tau-1}{5} \right)  \right) \\
=&
\nu_{\eta}(M_1) 
\left(
5c \cdot 
\frac{\tau+1}{5}
+(d-c)
\right)^{\frac12}
\eta \left(\frac{\tau+1}{5} \right) 
\nu_{\eta} (M_2)
\left(
5c \cdot 
\frac{\tau-1}{5}
+(d+c)
\right)^{\frac12}
\eta\left(\frac{\tau-1}{5} \right)  \\
=&
\nu_{\eta}(M_1) \nu_{\eta} (M_2)
(c\tau+d)H(\tau),
\end{align*}
where 
\begin{equation*}
M_1
=
\begin{pmatrix}
a-c & \frac{ b-d-a+c }{5} \\
5c & d-c
\end{pmatrix}, 
M_2
=
\begin{pmatrix}
a+c & \frac{ b+d+a+c }{5} \\
5c & d+c
\end{pmatrix}. 
\end{equation*}
\end{proof}

\subsection{The case where 
$
M
=
\begin{pmatrix}
a & b \\
c & d
\end{pmatrix}
\equiv \pm
\begin{pmatrix}
1 & 0 \\
0 & 1
\end{pmatrix} 
\bmod{5}
$ 
and 
$c$ is odd
}

\begin{theorem}
\label{thm:M=pmI-c-odd}
{\it
Suppose that 
$
M=
\begin{pmatrix}
a & b \\
c & d
\end{pmatrix}
\in 
\Gamma_{\theta,5}
$ 
and 
$M\equiv 
\pm
\begin{pmatrix}
1 & 0 \\
0 & 1
\end{pmatrix} 
\bmod{3}
$ 
and 
$c$ is odd. 
Then, 
the multiplier system 
$\nu_{H}$ of 
$
\displaystyle 
H=
\eta \left(\frac{\tau-1}{5} \right) \eta\left(\frac{\tau+1}{5} \right)
$ 
is given by 
\begin{equation*}
\nu_{H}
(M)
=
\exp
\left\{
\frac{\pi i}{6}
\left[
-3c-bd\cdot \frac{c^2-1}{2}-\frac{ac}{5}\cdot \frac{5c^2-7}{2}+\frac{cd}{5}(2c^2-1)
\right]
\right\}. 
\end{equation*}
}
\end{theorem}

\begin{proof}
We first note that 
\begin{align*}
H(M\tau)=&
\nu_{\eta}(M_1) \nu_{\eta} (M_2)
(c\tau+d)H(\tau),
\end{align*}
where 
\begin{equation*}
M_1
=
\begin{pmatrix}
a-c & \frac{ b-d+a-c }{5} \\
5c & d+c
\end{pmatrix}, 
M_2
=
\begin{pmatrix}
a+c & \frac{ b+d-a-c }{5} \\
5c & d-c
\end{pmatrix}. 
\end{equation*}
Equation (\ref{eqn-eta-multiplier})
yields 
\begin{align*}
\nu_{\eta}(M_1)=&
\left(
\frac{d+c}{5c}
\right)^{*}  
\exp
\Big\{
\frac{\pi i}{12}
\Big[
-5 a c^3-5 a c^2 d-5 b c^3-5 b c^2 d+5 c^4+10 c^3 d+5 c^2 d^2-15 c\\
&\hspace{40mm} 
+\frac{26 a c}{5}+\frac{a d}{5}+\frac{b c}{5}+\frac{b d}{5} -\frac{c^2}{5}+\frac{23 c d}{5}-\frac{d^2}{5}
\Big]
\Big\},
\end{align*}
and 
\begin{align*}
\nu_{\eta}(M_2)=&
\left(
\frac{d-c}{5c}
\right)^{*}  
\exp 
\Big\{
\frac{\pi i}{12}
\Big[
-5 a c^3+5 a c^2 d+5 b c^3-5 b c^2 d-5 c^4+10 c^3 d-5 c^2 d^2-15 c\\
&\hspace{40mm} 
 +\frac{26 a c}{5}-\frac{a d}{5}-\frac{b c}{5}+\frac{b d}{5}    
   +\frac{c^2}{5}+\frac{23 c d}{5}+\frac{d^2}{5}
\Big]
\Big\},
\end{align*}
which implies that 
\begin{equation*}
\nu_{F}
=
\nu_{\eta}(M_1) \nu_{\eta} (M_2)
=
\left(
\frac{d+c}{5c}
\right)^{*}  
\left(
\frac{d-c}{5c}
\right)^{*}  
\exp
\left\{
\frac{\pi i}{12}
E
\right\}, 
\end{equation*}
where 
\begin{equation*}
E=
-30 c+\frac{52 a c}{5}-10 a c^3+\frac{2 b d}{5}+\frac{46 c d}{5}-10 b c^2 d+20 c^3 d. 
\end{equation*}
\par
Since $c\equiv 0 \bmod{5},$ it follows that 
\begin{equation*}
\left(
\frac{d+c}{5c}
\right)^{*}  
\left(
\frac{d-c}{5c}
\right)^{*}  
=
\left(
\frac{d+c}{5 |c|}
\right) 
\left(
\frac{d-c}{5 |c|}
\right)=
\left(
\frac{d^2-c^2}{5 |c|}
\right)=
 \left(
\frac{d^2}{5 |c|}
\right)=1. 
\end{equation*} 
\par
Since $c$ is odd, 
it follows that 
$-9c^2\equiv 9 \bmod{24},$ 
which implies that 
\begin{align*}
E=&-30 c+\frac{52 a c}{5}-10 a c^3+\frac{2 b d}{5}+\frac{46 c d}{5}-10 b c^2 d+20 c^3 d  & &\\
\equiv& 
-6c+\frac45 ac-ac^3-9ac^3+\frac25 bd-\frac25 cd-bc^2d-9bc^2d-4c^3d & &\bmod{24} \\
\equiv&
-6c+\frac45 ac-ac^3-9ac+\frac25 bd-\frac25 cd-bc^2d-9bd-4c^3d  & &\bmod{24} \\
\equiv&-6c-\frac{41}{5}ac-ac^3-\frac{43}{5} bd -\frac25 cd  -bc^2d -4c^3d                    &   &\bmod{24}  \\
\equiv&
-6c+\frac{7}{5}ac-ac^3+bd -\frac25 cd  -bc^2d +\frac45 c^3d                    &   &\bmod{24} \\
=&
2\left[
-3c-bd\cdot \frac{c^2-1}{2}-\frac{ac}{5}\cdot \frac{5c^2-7}{2}+\frac{cd}{5}(2c^2-1)
\right], 
\end{align*}
which proves the theorem. 
\end{proof}

\subsection{The case where 
$
M
=
\begin{pmatrix}
a & b \\
c & d
\end{pmatrix}
\equiv \pm
\begin{pmatrix}
0 & -1 \\
1 & 0
\end{pmatrix} 
\bmod{5}
$ 
and 
$c$ is odd
}

\begin{theorem}
\label{thm:M=pmT-c-odd}
{\it
Suppose that 
$
M=
\begin{pmatrix}
a & b \\
c & d
\end{pmatrix}
\in 
\Gamma_{\theta,5}
$ 
and 
$M\equiv 
\pm
\begin{pmatrix}
0 & -1 \\
1 & 0
\end{pmatrix} 
\bmod{5}
$ 
and 
$c$ is odd. 
Then, 
the multiplier system 
$\nu_{H}$ of 
$
\displaystyle 
H=
\eta \left(\frac{\tau-1}{5} \right) \eta\left(\frac{\tau+1}{5} \right)
$ 
is given by 
\begin{equation*}
\nu_{H}
(M)
=
\exp
\left\{
\frac{\pi i}{6}
\left[
-3c-bd\cdot\frac{c^2-1}{2}-\frac{ac}{5}\cdot\frac{5c^2-7}{2}-\frac{cd}{5}(2c^2-3)
\right]
\right\}. 
\end{equation*}
In particular, we have $\nu_{H}(T)=-i,$ where 
$
T=
\begin{pmatrix}
0 & -1 \\
1 & 0
\end{pmatrix}.  
$
}
\end{theorem}

\begin{proof}
We first note that 
\begin{align*}
H(M\tau)=&
\nu_{\eta}(M_1) \nu_{\eta} (M_2)
(c\tau+d)H(\tau),
\end{align*}
where 
\begin{equation*}
M_1
=
\begin{pmatrix}
a-c & \frac{ b-d-a+c }{5} \\
5c & d-c
\end{pmatrix}, 
M_2
=
\begin{pmatrix}
a+c & \frac{ b+d+a+c }{5} \\
5c & d+c
\end{pmatrix}. 
\end{equation*}
Equation (\ref{eqn-eta-multiplier})
yields 
\begin{align*}
\nu_{\eta}(M_1)=&
\left(
\frac{d-c}{5c}
\right)^{*}  
\exp
\Big\{
\frac{\pi i}{12}
\Big[
-5 a c^3+5 a c^2 d+5 b c^3-5 b c^2 d+5 c^4-10 c^3 d+5 c^2 d^2-15 c \\
&\hspace{40mm}
-\frac{b c}{5}+\frac{b d}{5}+\frac{26 a c}{5}-\frac{a d}{5}-\frac{51 c^2}{5}+\frac{27 c d}{5}-\frac{d^2}{5}
\Big]
\Big\},
\end{align*}
and 
\begin{align*}
\nu_{\eta}(M_2)=&
\left(
\frac{d+c}{3c}
\right)^{*}  
\exp 
\Big\{
\frac{\pi i}{12}
\Big[
-5 a c^3-5 a c^2 d-5 b c^3-5 b c^2 d-5 c^4-10 c^3 d-5 c^2 d^2-15 c \\
&\hspace{40mm}
+\frac{b c}{5}+\frac{b d}{5}+\frac{26 a c}{5}+\frac{a d}{5}+\frac{51 c^2}{5}+\frac{27 c d}{5}+\frac{d^2}{5}
\Big]
\Big\},
\end{align*}
which implies that 
\begin{equation*}
\nu_{H}
=
\nu_{\eta}(M_1) \nu_{\eta} (M_2)
=
\left(
\frac{d-c}{5c}
\right)^{*}  
\left(
\frac{d+c}{5c}
\right)^{*}  
\exp
\left\{
\frac{\pi i}{12}
E
\right\}, 
\end{equation*}
where 
\begin{equation*}
E=
-30 c+\frac{52 a c}{5}-10 a c^3+\frac{2 b d}{5}+\frac{54 c d}{5}-10 b c^2 d-20 c^3 d
\end{equation*}
\par
Since $c\equiv \pm1 \bmod{5}$ and $d\equiv0 \bmod{5},$ it follows that 
\begin{align*}
\left(
\frac{d-c}{5c}
\right)^{*}  
\left(
\frac{d+c}{5c}
\right)^{*}  
=&
\left(
\frac{d-c}{5 |c|}
\right) 
\left(
\frac{d+c}{5 |c|}
\right)=
\left(
\frac{d-c}{5}
\right)
\left(
\frac{d-c}{ |c|}
\right) 
\left(
\frac{d+c}{5}
\right)
\left(
\frac{d+c}{ |c|}
\right)  \\
=&\left(
\frac{d^2-c^2}{5}
\right)
\left(
\frac{d^2-c^2}{ |c|}
\right) 
=
\left(
\frac{-1}{5}
\right)
\left(
\frac{d^2}{ |c|}
\right) 
=1.
\end{align*} 
\par
Since $c$ is odd, 
it follows that 
$-9c^2\equiv -9 \bmod{24},$ 
which implies that 
\begin{align*}
E=&-30 c+\frac{52 a c}{5}-10 a c^3+\frac{2 b d}{5}+\frac{54 c d}{5}-10 b c^2 d-20 c^3 d  &  &   \\
\equiv& 
-6c+\frac45 ac-ac^3-9ac^3+\frac25 bd +\frac65 cd -9 bc^2d-bc^2d+4c^3d  &  &\bmod{24} \\
\equiv&
-6c+\frac45 ac-ac^3-9ac+\frac25 bd+\frac65 cd-9bd-bc^2d+4c^3d    &  &\bmod{24} \\
\equiv&
-6c-\frac{41}{5} ac-ac^3-\frac{43}{5} bd +\frac65 cd-bc^2d+4c^3d &  &\bmod{24} \\
\equiv&
-6c+\frac{7}{5} ac-ac^3+ bd +\frac65 cd-bc^2d-\frac45 c^3d &  &\bmod{24} \\
=&
2
\left[
-3c-bd\cdot\frac{c^2-1}{2}-\frac{ac}{5}\cdot\frac{5c^2-7}{2}-\frac{cd}{5}(2c^2-3)
\right], 
\end{align*}
which proves the theorem. 
\end{proof}

\subsection{The case where 
$
M
=
\begin{pmatrix}
a & b \\
c & d
\end{pmatrix}
\equiv \pm
\begin{pmatrix}
1 & 0 \\
0 & 1
\end{pmatrix} 
\bmod{5}
$ 
and 
$c$ is even
}

\begin{theorem}
\label{thm:M=pmI-c-even}
{\it
Suppose that 
$
M=
\begin{pmatrix}
a & b \\
c & d
\end{pmatrix}
\in 
\Gamma_{\theta,5}
$ 
and 
$M\equiv 
\pm
\begin{pmatrix}
1 & 0 \\
0 & 1
\end{pmatrix} 
\bmod{5}
$ 
and 
$c$ is even. 
Then, 
the multiplier system 
$\nu_{H}$ of 
$
\displaystyle 
H=
\eta \left(\frac{\tau-1}{5} \right) \eta\left(\frac{\tau+1}{5} \right)
$ 
is given by 
\begin{equation*}
\nu_{H}
(M)
=
\exp
\left\{
\frac{\pi i}{6}
\left[
3-3d-ac\cdot\frac{d^2-1}{2}-\frac{bd}{5}\cdot\frac{5d^2-7}{2}+\frac{cd}{5}(2d^2-3)
\right]
\right\}. 
\end{equation*}
}
\end{theorem}

\begin{proof}
Since $c$ is even, it follows that $d$ is odd. 
By 
$
T
=
\begin{pmatrix}
0 & -1 \\
1 & 0
\end{pmatrix}, 
$ 
we have 
\begin{equation*}
MT
=
\begin{pmatrix}
a & b \\
c & d
\end{pmatrix}
\begin{pmatrix}
0 & -1 \\
1 & 0
\end{pmatrix}
=
\begin{pmatrix}
b & -a \\
d & -c
\end{pmatrix}
\equiv
\pm
\begin{pmatrix}
0 & -1 \\
1 & 0
\end{pmatrix}
\bmod{5}.
\end{equation*}
\par
Proposition \ref{prop:character-level5} and Theorem \ref{thm:M=pmT-c-odd} 
imply that 
\begin{align*}
\nu_{H}(MT)=&\nu_{H}(M)\nu_{H}(T)=\nu_{H}(M)\cdot(-i)  \\
=&
\exp
\left\{
\frac{\pi i}{6}
\left[
-3d-ac\cdot\frac{d^2-1}{2}-\frac{bd}{5}\cdot\frac{5d^2-7}{2}-\frac{-cd}{5}(2d^2-3)
\right]
\right\},
\end{align*}
which proves the theorem. 
\end{proof}

\subsection{The case where 
$
M
=
\begin{pmatrix}
a & b \\
c & d
\end{pmatrix}
\equiv \pm
\begin{pmatrix}
0 & -1 \\
1 & 0
\end{pmatrix} 
\bmod{5}
$ 
and 
$c$ is even
}

\begin{theorem}
\label{thm:M=pmT-c-even}
{\it
Suppose that 
$
M=
\begin{pmatrix}
a & b \\
c & d
\end{pmatrix}
\in 
\Gamma_{\theta,3}
$ 
and 
$M\equiv 
\pm
\begin{pmatrix}
0 & -1 \\
1 & 0
\end{pmatrix} 
\bmod{5}
$ 
and 
$c$ is even. 
Then, 
the multiplier system 
$\nu_{H}$ of 
$
\displaystyle 
H=
\eta \left(\frac{\tau-1}{5} \right) \eta\left(\frac{\tau+1}{5} \right)
$ 
is given by 
\begin{equation*}
\nu_{H}
(M)
=
\exp
\left\{
\frac{\pi i}{6}
\left[
3-3d-ac\cdot \frac{d^2-1}{2}-\frac{bd}{5}\cdot \frac{5d^2-7}{2}-\frac{cd}{5}(2d^2-1)
\right]
\right\}. 
\end{equation*}
}
\end{theorem}

\begin{proof}
Since $c$ is even, it follows that $d$ is odd. 
By 
$
T
=
\begin{pmatrix}
0 & -1 \\
1 & 0
\end{pmatrix}, 
$ 
we have 
\begin{equation*}
MT
=
\begin{pmatrix}
a & b \\
c & d
\end{pmatrix}
\begin{pmatrix}
0 & -1 \\
1 & 0
\end{pmatrix}
=
\begin{pmatrix}
b & -a \\
d & -c
\end{pmatrix}
\equiv
\mp
\begin{pmatrix}
1 & 0 \\
0 & 1
\end{pmatrix}
\bmod{5}.
\end{equation*}
\par
Proposition \ref{prop:character-level5} and Theorem \ref{thm:M=pmI-c-odd} 
imply that 
\begin{align*}
\nu_{H}(MT)=&\nu_{H}(M)\nu_{H}(T)=\nu_{H}(M)\cdot(-i)  \\
=&
\exp
\left\{
\frac{\pi i}{6}
\left[
-3d-ac\cdot \frac{d^2-1}{2}-\frac{bd}{5}\cdot \frac{5d^2-7}{2}+\frac{-cd}{5}(2d^2-1)
\right]
\right\},
\end{align*}
which proves the theorem. 
\end{proof}

\subsection{Summary}

\begin{theorem}
\label{summary-level5}
{\it
Set
$
M=
\begin{pmatrix}
a & b \\
c & d
\end{pmatrix}
\in 
\Gamma_{\theta,5}. 
$ 
Then, it follows that 
the multiplier system 
$\nu_{H}$ of 
 $
\displaystyle 
H=
\eta \left(\frac{\tau-1}{5} \right) \eta\left(\frac{\tau+1}{5} \right) 
$ 
is given by
$
\displaystyle
\exp
\left\{
\frac{\pi i}{6}
h(M)
\right\},
$
where $h(M)$ is 
\begin{align*}
&h(M)\\
=&
\begin{cases}
\displaystyle
-3c-bd\cdot \frac{c^2-1}{2}-\frac{ac}{5} \cdot\frac{5c^2-7}{2}+\frac{cd}{5}(2c^2-1) 
&
\text
{ 
if
$
M
\equiv 
\pm
\begin{pmatrix}
1 & 0 \\
0 & 1
\end{pmatrix} 
\bmod{5}
$ 
and 
$c$ is odd,
} \vspace{0.5mm} \\
\displaystyle
3-3d-ac\cdot\frac{d^2-1}{2}-\frac{bd}{5} \cdot \frac{5d^2-7}{2}+\frac{cd}{5}(2d^2-3)
&
\text
{ 
if
$
M
\equiv 
\pm
\begin{pmatrix}
1 & 0 \\
0 & 1
\end{pmatrix} 
\bmod{5}
$ 
and 
$c$ is even,
}   \vspace{0.5mm}  \\
\displaystyle
-3c-bd\cdot \frac{c^2-1}{2}-\frac{ac}{5} \cdot \frac{5c^2-7}{2}-\frac{cd}{5}(2c^2-3)
&
\text
{ 
if
$
M
\equiv 
\pm
\begin{pmatrix}
0 & -1 \\
1 & 0
\end{pmatrix} 
\bmod{5}
$ 
and 
$c$ is odd,
}   \vspace{0.5mm}  \\
\displaystyle
3-3d-ac\cdot  \frac{d^2-1}{2}-\frac{bd}{5} \cdot \frac{5d^2-7}{2}-\frac{cd}{5}(2d^2-1)
&
\text
{ 
if
$
M
\equiv 
\pm
\begin{pmatrix}
0 & -1 \\
1 & 0
\end{pmatrix} 
\bmod{5}
$ 
and 
$c$ is even.
} 
\end{cases}
\end{align*}
}
\end{theorem}

\begin{proof}
The theorem follows from 
Theorems 
\ref{thm:M=pmI-c-odd}, 
\ref{thm:M=pmT-c-odd}, 
\ref{thm:M=pmI-c-even} 
and 
\ref{thm:M=pmT-c-even}. 
\end{proof}

From the definition, 
it follows that for each $k\in\mathbb{Z},$ 
\begin{equation*}
\nu_{H^k}(M)=
\exp \left\{
\frac{\pi i k}{6}
h(M)
\right\}. 
\end{equation*}

\subsection{The case where $k\equiv 0 \bmod{12}$}

\begin{theorem}
{\it
Suppose that 
$k\equiv 0 \bmod{12}$. 
Then, it follows that 
\begin{equation*}
\Ker \nu_{H^{k}}
=\Gamma_{\theta,5}.
\end{equation*}
}
\end{theorem}

\begin{proof}
It is obvious. 
\end{proof}

\subsection{The case where $k\equiv 6 \bmod{12}$}

\begin{theorem}
\label{thmk=6-level5}
{\it
Suppose that 
$k\equiv 6 \bmod{12}$. 
Then, it follows that 
\begin{equation*}
\Ker \nu_{H^{k}}
=
\left\{
\begin{pmatrix}
a & b \\
c & d
\end{pmatrix}
\in
\Gamma_{\theta,5} \,\Big| \, 
\begin{pmatrix}
a & b \\
c & d
\end{pmatrix}
\equiv
\begin{pmatrix}
1 & 0 \\
0 & 1
\end{pmatrix}, 
\begin{pmatrix}
0 & 1 \\
1 & 1
\end{pmatrix},
\begin{pmatrix}
1 & 1 \\
1 & 0
\end{pmatrix}
\bmod{2}
\right\}.
\end{equation*}
Moreover, 
as a coset decomposition of $\Gamma_{\theta,5}$ modulo $\Ker \nu_{H^{k}}$,  
we may choose $S^{n}$ $(n=0,5),$ where 
$
S=
\begin{pmatrix}
1 & 1 \\
0 & 1
\end{pmatrix}. 
$
}
\end{theorem}

\begin{proof}
For each 
$
M
=
\begin{pmatrix}
a & b \\
c & d
\end{pmatrix}
\in
\Gamma_{\theta,5}, 
$ 
we have 
\begin{equation*}
\nu_{H^{k}  }(M)=
\exp
\{\pi i h(M) \}. 
\end{equation*}
\par
If $M \equiv \pm
\begin{pmatrix}
1 & 0 \\
0 & 1
\end{pmatrix}
\bmod{5}
$ 
and $c$ is odd, 
we have 
\begin{align*}
h(M)\equiv 0 \bmod{2}
\Longleftrightarrow& \,
-3c+\frac15 (a+d)c \equiv 0 \bmod{2} \\
\Longleftrightarrow& \,
a+d\equiv 1 \bmod{2} \\
\Longleftrightarrow& \,
(a,d)\equiv (1,0), (0,1) \bmod{2} \\
\Longleftrightarrow& \,
\begin{pmatrix}
a & b \\
c & d
\end{pmatrix}
\equiv 
\begin{pmatrix}
1 & 1 \\
1 & 0
\end{pmatrix},
\begin{pmatrix}
0 & 1 \\
1 & 1
\end{pmatrix}
\bmod{2}. 
\end{align*}
\par
If $M \equiv \pm
\begin{pmatrix}
1 & 0 \\
0 & 1
\end{pmatrix}
\bmod{5}
$ 
and $c$ is even, 
we have 
\begin{align*}
h(M)\equiv 0 \bmod{2}
\Longleftrightarrow& \,
3-3d-\frac{bd}{5}\cdot\frac{5d^2-7}{5}
\equiv 0 \bmod{2} \\
\Longleftrightarrow& \,
d-1+bd \equiv 0 \bmod{2} \\
\Longleftrightarrow& \,
(b,d)\equiv (0,1) \bmod{2} \\
\Longleftrightarrow& \,
\begin{pmatrix}
a & b \\
c & d
\end{pmatrix}
\equiv 
\begin{pmatrix}
1 & 0 \\
0 & 1
\end{pmatrix}
\bmod{2}. 
\end{align*}
\par
If $M \equiv \pm
\begin{pmatrix}
0 & -1 \\
1 & 0
\end{pmatrix}
\bmod{5}
$ 
and $c$ is odd, 
we have 
\begin{align*}
h(M)\equiv 0 \bmod{2}
\Longleftrightarrow& \,
-3c-\frac{ac}{5}-\frac{cd}{5} \equiv 0 \bmod{2} \\
\Longleftrightarrow& \,
a+d\equiv 1 \bmod{2} \\
\Longleftrightarrow& \,
(a,d)\equiv (1,0), (0,1) \bmod{2} \\
\Longleftrightarrow& \,
\begin{pmatrix}
a & b \\
c & d
\end{pmatrix}
\equiv 
\begin{pmatrix}
1 & 1 \\
1 & 0
\end{pmatrix},
\begin{pmatrix}
0 & 1 \\
1 & 1
\end{pmatrix}
\bmod{2}. 
\end{align*}
\par
If $M \equiv \pm
\begin{pmatrix}
0 & -1 \\
1 & 0
\end{pmatrix}
\bmod{5}
$ 
and $c$ is even, 
we have 
\begin{align*}
h(M)\equiv 0 \bmod{2}
\Longleftrightarrow& \,
3-3d-\frac{bd}{5}-\frac{cd}{5} \equiv 0 \bmod{2} \\
\Longleftrightarrow& \,
(b,d)\equiv  (0,1) \bmod{2} \\
\Longleftrightarrow& \,
\begin{pmatrix}
a & b \\
c & d
\end{pmatrix}
\equiv 
\begin{pmatrix}
1 & 0 \\
0 & 1
\end{pmatrix}
\bmod{2}. 
\end{align*}
\par
The coset decomposition can be obtained by considering the values of $\nu_{H^{k}  }(M).$
\end{proof}

\subsection{The case where $k\equiv \pm 3 \bmod{12}$}

\begin{lemma}
\label{lem:level5-mod4}
{\it
Set 
$
M=
\begin{pmatrix}
a & b \\
c & d
\end{pmatrix}
\in\Gamma_{\theta,5}.
$
\newline
$\mathrm{(1)}$ If 
$
M
\equiv
\pm
\begin{pmatrix}
1 & 0 \\
0 & 1
\end{pmatrix}
\bmod{5}
$ 
and $c$ is odd, 
we have 
\begin{align}
h(M)\equiv 0 \bmod{4} 
\Longleftrightarrow&
\, a+d\equiv -1 \bmod{4} \notag  \\
\Longleftrightarrow&
M
\equiv
\begin{pmatrix}
0 & -1 \\
1 & -1
\end{pmatrix}, 
\begin{pmatrix}
0 & 1 \\
-1 & -1
\end{pmatrix}, 
\begin{pmatrix}
1 & 1 \\
1 & 2
\end{pmatrix}, 
\begin{pmatrix}
1 & -1 \\
-1 & 2
\end{pmatrix},  \notag \\
&\hspace{10mm}
\begin{pmatrix}
2 & 1 \\
1 & 1
\end{pmatrix}, 
\begin{pmatrix}
2 & -1 \\
-1 & 1
\end{pmatrix}, 
\begin{pmatrix}
-1 & -1 \\
1 & 0
\end{pmatrix}, 
\begin{pmatrix}
-1 & 1 \\
-1 & 0
\end{pmatrix} \bmod{4}.   \label{eqn:level5-mod4-pmI-c-odd}
\end{align}
$\mathrm{(2)}$ If 
$
M
\equiv
\pm
\begin{pmatrix}
1 & 0 \\
0 & 1
\end{pmatrix}
\bmod{5}
$ 
and $c$ is even, we have  
\begin{align}
 \label{eqn:level5-mod4-pmI-c-even}
h(M)\equiv 0 \bmod{4} 
\Longleftrightarrow& 
\, b-c\equiv d-1 \bmod{4} \notag \\
\Longleftrightarrow& \, 
M
\equiv 
\begin{pmatrix}
1 & 0 \\
0 & 1
\end{pmatrix}, 
\begin{pmatrix}
1 & 2 \\
2 & 1
\end{pmatrix}, 
\begin{pmatrix}
-1 & 0 \\
2 & -1
\end{pmatrix}, 
\begin{pmatrix}
-1 & 2 \\
0 & -1
\end{pmatrix} 
\bmod{4}.
\end{align}
$\mathrm{(3)}$ If 
$
M
\equiv
\pm
\begin{pmatrix}
0 & -1 \\
1 & 0
\end{pmatrix}
\bmod{5}
$ 
and $c$ is odd, 
we have 
\begin{align}
h(M)\equiv 0 \bmod{4} 
\Longleftrightarrow&
\, a+d\equiv -1 \bmod{4}  \notag \\
\Longleftrightarrow&
M
\equiv
\begin{pmatrix}
0 & 1 \\
-1 & -1
\end{pmatrix}, 
\begin{pmatrix}
0 & -1 \\
1 & -1
\end{pmatrix}, 
\begin{pmatrix}
-1 & 1 \\
-1 & 0
\end{pmatrix}, 
\begin{pmatrix}
-1 & -1 \\
1 & 0
\end{pmatrix},  \notag \\
&\hspace{10mm}
\begin{pmatrix}
2 & 1 \\
1 & 1
\end{pmatrix}, 
\begin{pmatrix}
2 & -1 \\
-1 & 1
\end{pmatrix}, 
\begin{pmatrix}
1 & 1 \\
1 & 2
\end{pmatrix}, 
\begin{pmatrix}
1 & -1 \\
-1 & 2
\end{pmatrix} \bmod{4}.   \label{eqn:level5-mod4-pmT-c-odd}
\end{align}
$\mathrm{(4)}$ If 
$
M
\equiv
\pm
\begin{pmatrix}
0 & -1 \\
1 & 0
\end{pmatrix}
\bmod{5}
$ 
and $c$ is even, we have  
\begin{align}
\label{eqn:level3-mod4-pmT-c-even}
h(M)\equiv 0 \bmod{4} 
\Longleftrightarrow& \,
b-c\equiv d-1 \bmod{4} \notag \\
\Longleftrightarrow&
M
\equiv 
\begin{pmatrix}
1 & 0 \\
0 & 1
\end{pmatrix}, 
\begin{pmatrix}
1 & 2 \\
2 & 1
\end{pmatrix}, 
\begin{pmatrix}
-1 & 0 \\
2 & -1
\end{pmatrix}, 
\begin{pmatrix}
-1 & 2 \\
0 & -1
\end{pmatrix} 
\bmod{4}.
\end{align}
}
\end{lemma}

\begin{proof}
We treat cases (1) and (2). The other cases can be proved in the same way. 
\newline
{\bf Case (1).}
From the definition of $h(M),$ it follows that 
\begin{align*}
h(M)\equiv 0 \bmod{4}
\Longleftrightarrow&\, 
-3c+\frac{ac}{5}+\frac{cd}{5}\equiv 0 \bmod{4}  \\
\Longleftrightarrow&\, 
a+d\equiv -1 \bmod{4} \\
\Longleftrightarrow&\, 
(a,d)\equiv (0,-1), (1,2), (2,1), (-1,0) \bmod{4}  \\
\Longleftrightarrow&\, 
M
\equiv
\begin{pmatrix}
0 & -1 \\
1 & -1
\end{pmatrix}, 
\begin{pmatrix}
0 & 1 \\
-1 & -1
\end{pmatrix}, 
\begin{pmatrix}
1 & 1 \\
1 & 2
\end{pmatrix}, 
\begin{pmatrix}
1 & -1 \\
-1 & 2
\end{pmatrix},  \\
&\hspace{10mm}
\begin{pmatrix}
2 & 1 \\
1 & 1
\end{pmatrix}, 
\begin{pmatrix}
2 & -1 \\
-1 & 1
\end{pmatrix}, 
\begin{pmatrix}
-1 & -1 \\
1 & 0
\end{pmatrix}, 
\begin{pmatrix}
-1 & 1 \\
-1 & 0
\end{pmatrix} \bmod{4}. 
\end{align*}
{\bf Case (2).}
From the definition of $h(M),$ it follows that 
\begin{align*}
h(M)\equiv 0 \bmod{4}
\Longleftrightarrow&\, 
3-3d+\frac{bd}{5}-\frac{cd}{5}
 \equiv 0 \bmod{4}  \\
\Longleftrightarrow&\, 
b-c\equiv d-1.
\end{align*}
If $d\equiv 1 \bmod{4},$ 
we have 
\begin{align*}
b-c-3\equiv 1 \bmod{4}
\Longleftrightarrow&\,
(b,c)\equiv (0,0), (2,2) \bmod{4}   \\
\Longleftrightarrow&\,
M
\equiv
\begin{pmatrix}
1 & 0\\
0 & 1
\end{pmatrix}, 
\begin{pmatrix}
1 & 2\\
2 & 1
\end{pmatrix} \bmod{4}. 
\end{align*}
If $d\equiv -1 \bmod{4},$ 
we have 
\begin{align*}
b-c-3\equiv -1 \bmod{4}
\Longleftrightarrow&\,
(b,c)\equiv (0,2), (2,0) \bmod{4}   \\
\Longleftrightarrow&\,
M
\equiv
\begin{pmatrix}
-1 & 0\\
2 & -1
\end{pmatrix}, 
\begin{pmatrix}
-1 & 2\\
0 & -1
\end{pmatrix} \bmod{4}. 
\end{align*}
\end{proof}

\begin{theorem}
\label{thmk=3-level5}
{\it
Suppose that 
$k\equiv \pm 3 \bmod{12}$. 
Then, it follows that 
\begin{align*}
\Ker \nu_{H^{k}} 
=
\Big\{
M=
\begin{pmatrix}
a & b \\
c & d
\end{pmatrix}
\in
\Gamma_{\theta,5} \,\Big| \, 
  &a+d \equiv& -1 &  &\bmod{4}  &  &(c\equiv 0 \bmod{5}, c\equiv 1 \bmod{2}), & & \\
  &b-c \equiv& d-1&  &\bmod{4}  & &(c\equiv 0 \bmod{5}, c \equiv 0 \bmod{2}), & & \\
  &a+d \equiv& -1 &  &\bmod{4}  &  &(c \not\equiv 0 \bmod{5}, c \equiv 1 \bmod{2}), & &  \\
  &b-c \equiv& d-1 &  &\bmod{4} & &(c\not\equiv 0 \bmod{5}, c\equiv 0 \bmod{2}) & &\hspace{-4mm}
\Big\}.
\end{align*}
Moreover, 
as a coset decomposition of $\Gamma_{\theta,5}$ modulo $\Ker \nu_{H^{k}}$,  
we may choose $S^{n}$ $(n=0,5,10,15),$ where 
$
S=
\begin{pmatrix}
1 & 1 \\
0 & 1
\end{pmatrix}. 
$
}
\end{theorem}

\begin{proof}
For each $M\in\Gamma_{\theta,5},$ 
we have 
\begin{equation*}
\nu_{H^k}(M)=
\exp
\left\{
\pm
\frac{\pi i}{2}
\left[
f(M)
\right]
\right\},
\end{equation*}
which proves the theorem. 
The coset decomposition can be obtained by considering the values of $\nu_{H^{k}  }(M).$
\end{proof}

\subsection{The case where $k\equiv \pm 4 \bmod{12}$}

\begin{lemma}
\label{lem:level5-mod3}
{\it
Set 
$
M=
\begin{pmatrix}
a & b \\
c & d
\end{pmatrix}
\in\Gamma_{\theta,5}.
$
\newline
$\mathrm{(1)}$ If 
$
M
\equiv
\pm
\begin{pmatrix}
1 & 0 \\
0 & 1
\end{pmatrix}
\bmod{5},
$ 
we have 
\begin{align*}
h(M)\equiv 0 \bmod{3} 
\Longleftrightarrow
M
\equiv& 
\pm
\begin{pmatrix}
1 & 0 \\
0 & 1
\end{pmatrix}, 
\pm
\begin{pmatrix}
0 & -1 \\
1 & 0
\end{pmatrix}, 
\pm
\begin{pmatrix}
1 & 1 \\
1 & -1
\end{pmatrix},  
\pm
\begin{pmatrix}
1 & -1 \\
-1 & -1
\end{pmatrix}
\bmod{3}.   
\end{align*}
\newline
$\mathrm{(2)}$ If 
$
M
\equiv
\pm
\begin{pmatrix}
0 & -1 \\
1 & 0
\end{pmatrix}
\bmod{5},
$ 
we have 
\begin{align*}
h(M)\equiv 0 \bmod{3} 
\Longleftrightarrow
M
\equiv& 
\pm
\begin{pmatrix}
1 & 0 \\
0 & 1
\end{pmatrix}, 
\pm
\begin{pmatrix}
0 & -1 \\
1 & 0
\end{pmatrix}, 
\pm
\begin{pmatrix}
1 & 1 \\
1 & -1
\end{pmatrix},  
\pm
\begin{pmatrix}
1 & -1 \\
-1 & -1
\end{pmatrix}
\bmod{3}.   
\end{align*}
}
\end{lemma}

\begin{proof}
We treat case (1). The other case can be proved in the same way. 
\par
Suppose that $c$ is odd. 
From the definition of $h(M),$ it follows that 
\begin{equation*}
h(M)\equiv 0 \bmod{3}
\Longleftrightarrow
bd(c^2-1)+ac(c^2+1)+cd(c^2+1) \equiv 0 \bmod{3}. 
\end{equation*}
If $c\equiv 0,$ we have $bd\equiv0 \bmod{3},$ which implies that 
\begin{equation*}
M
\equiv 
\pm
\begin{pmatrix}
1 & 0 \\
0 & 1
\end{pmatrix} 
\bmod{3}.
\end{equation*}
If $c\not\equiv 0 \bmod{3},$ 
we have 
$a+d\equiv 0 \bmod{3},$ 
which implies that 
\begin{equation*}
(a,d)\equiv (0,0), \pm(1,-1) \bmod{3}.
\end{equation*}
Since $ad-bc=1,$ it follows that 
\begin{equation*}
M
\equiv 
\pm
\begin{pmatrix}
0 & -1 \\
1 & 0
\end{pmatrix}, 
\pm
\begin{pmatrix}
1 & 1 \\
1 & -1
\end{pmatrix},  
\pm
\begin{pmatrix}
1 & -1 \\
-1 & -1
\end{pmatrix}
\bmod{3}. 
\end{equation*}
\par
Suppose that $c$ is even. 
From the definition of $h(M),$ it follows that 
\begin{equation*}
h(M)\equiv 0 \bmod{3}
\Longleftrightarrow
ac(d^2-1)+bd(d^2+1)+cd^3 \equiv 0 \bmod{3}. 
\end{equation*}
If $d\equiv 0,$ we have $ac\equiv0 \bmod{3},$ which implies that 
\begin{equation*}
M
\equiv 
\pm
\begin{pmatrix}
0 & -1 \\
1 & 0
\end{pmatrix} 
\bmod{3}.
\end{equation*}
If $d\not\equiv 0 \bmod{3},$ 
we have 
$b-c\equiv 0 \bmod{3},$ 
which implies that 
\begin{equation*}
(b,c)\equiv (0,0), \pm(1,1) \bmod{3}.
\end{equation*}
Since $ad-bc=1,$ it follows that 
\begin{equation*}
M
\equiv 
\pm
\begin{pmatrix}
1 & 0 \\
0 & 1
\end{pmatrix}, 
\pm
\begin{pmatrix}
1 & 1 \\
1 & -1
\end{pmatrix},  
\pm
\begin{pmatrix}
1 & -1 \\
-1 & -1
\end{pmatrix}
\bmod{3}. 
\end{equation*}
\end{proof}

\begin{theorem}
\label{thmk=4-level5}
{\it
Suppose that 
$k\equiv \pm 4 \bmod{12}$. 
Then, it follows that 
\begin{align*}
\Ker \nu_{H^{k}}
=&
\Big\{
M
\in
\Gamma_{\theta, 5} \,\Big| \, 
M
\equiv& 
\pm
\begin{pmatrix}
1 & 0 \\
0 & 1
\end{pmatrix}, 
\pm
\begin{pmatrix}
0 & -1 \\
1 & 0
\end{pmatrix}, 
\pm
\begin{pmatrix}
1 & 1 \\
1 & -1
\end{pmatrix},  
\pm
\begin{pmatrix}
1 & -1 \\
-1 & -1
\end{pmatrix}
\bmod{3}
\Big\}.
\end{align*}
Moreover, 
as a coset decomposition of $\Gamma_{\theta,5}$ modulo $\Ker \nu_{H^{k}}$,  
we may choose $S^{n}$ $(n=0,5,10),$ where 
$
S=
\begin{pmatrix}
1 & 1 \\
0 & 1
\end{pmatrix}. 
$
}
\end{theorem}

\begin{proof}
For each $M\in\Gamma_{\theta,5},$ 
we have 
\begin{equation*}
\nu_{H^k}(M)
=
\exp
\left\{
\pm
\frac{2\pi i}{3}
\left[
h(M)
\right]
\right\}. 
\end{equation*}
The theorem follows from Lemma \ref{lem:level5-mod3}.  
The coset decomposition can be obtained by considering the values of $\nu_{H^{k}  }(M).$
\end{proof}

\subsection{The case where $k\equiv \pm 2 \bmod{12}$}

\begin{theorem}
\label{thmk=2-level5}
{\it
Suppose that 
$k\equiv \pm2 \bmod{12}$. 
Then, it follows that 
\begin{align*}
\Ker \nu_{H^{k}}
=
\Ker 
\nu_{ H^{6} }
\cap
\Ker \nu_{ H^4 }. 
\end{align*}
Moreover, 
as a coset decomposition of $\Gamma_{\theta,5}$ modulo $\Ker \nu_{H^{k}}$,  
we may choose 
\begin{equation*}
S^{n}  \,\, 
(n=0,5, 10, 15,20,25),  \,\,
S=
\begin{pmatrix}
1 & 1 \\
0 & 1
\end{pmatrix}. 
\end{equation*}
}
\end{theorem}

\begin{proof}
For each 
$
M
=
\begin{pmatrix}
a & b \\
c & d
\end{pmatrix}
\in
\Gamma_{\theta,5}, 
$ 
we have 
\begin{equation*}
\nu_{H^{k}  }(M)=
\exp
\left\{
\pm
\frac{
 \pi i
}
{3} 
h(M) 
\right\},
\end{equation*}
which implies that 
\begin{equation*}
M
\in
\Ker \nu_{H^{k}}
\Longleftrightarrow 
h(M)\equiv 0 \bmod{6}
\Longleftrightarrow 
h(M)\equiv 0 \bmod{2} \,\,
\mathrm{and} \,\,
h(M)\equiv 0 \bmod{3}. 
\end{equation*}
The theorem follows from Theorems \ref{thmk=6-level5} and \ref{thmk=4-level5}.
The coset decomposition can be obtained by considering the values of $\nu_{F^{k}  }(M).$
\end{proof}

\subsection{The case where $k\equiv \pm 1, \pm 5 \bmod{12}$}

\begin{theorem}
\label{thmk=1and5}
{\it
Suppose that 
$k\equiv \pm1, \pm 5 \bmod{12}$. 
Then, it follows that 
\begin{align*}
\Ker \nu_{H^{k}}
=
\Ker 
\nu_{ H^{3} }
\cap
\Ker \nu_{ H^4 }. 
\end{align*}
Moreover, 
as a coset decomposition of $\Gamma_{\theta,5}$ modulo $\Ker \nu_{H^{k}}$,  
we may choose 
\begin{equation*}
S^{n}  \,\, 
(n=0,5,10, 15, 20, 25, 30, 35, 40, 45, 50, 55),  \,\,
S=
\begin{pmatrix}
1 & 1 \\
0 & 1
\end{pmatrix}. 
\end{equation*}
}
\end{theorem}

\begin{proof}
For each 
$
M
=
\begin{pmatrix}
a & b \\
c & d
\end{pmatrix}
\in
\Gamma_{\theta,5}, 
$ 
we have 
\begin{equation*}
\nu_{ H^{k}  }(M)=
\exp
\left\{
\pm
\frac{
k
 \pi i
}
{6} 
h(M) 
\right\}, \,\,
(k,6)=1, 
\end{equation*}
which implies that 
\begin{equation*}
M
\in
\Ker \nu_{H^{k}}
\Longleftrightarrow 
h(M)\equiv 0 \bmod{12}
\Longleftrightarrow 
h(M)\equiv 0 \bmod{4} \,\,
\mathrm{and} \,\,
h(M)\equiv 0 \bmod{3}. 
\end{equation*}
The theorem follows from Theorems \ref{thmk=3-level5} and \ref{thmk=4-level5}.
The coset decomposition can be obtained by considering the values of $\nu_{ H^{k}  }(M).$
\end{proof}

\newpage

\section{Coset decomposition of $\Gamma(1) $ modulo $\Gamma_{\theta,5}$}
\label{sec:coset}

For
$
M
=
\begin{pmatrix}
a & b \\
c & d
\end{pmatrix}
\in
\Gamma(1),
$ 
we define an element $\lambda_N(M)$ of $SL(2, \mathbb{Z}/N \mathbb{Z} )$ by 
\begin{equation*}
\lambda_N(M)
=
\begin{pmatrix}
\bar{a} & \bar{b} \\
\bar{c} & \bar{d}
\end{pmatrix},
\end{equation*}
where 
$\bar{a} \equiv a \bmod{N}, $ 
$\bar{b} \equiv b \bmod{N}, $
$\bar{c} \equiv c \bmod{N}, $
$\bar{d} \equiv d \bmod{N}. $
In addition, 
we define 
\begin{equation*}
\Gamma(N)
=
\left\{
M
=
\begin{pmatrix}
a & b \\
c & d
\end{pmatrix}
\in
\Gamma(1) \, | \,
M\equiv 
 \begin{pmatrix}
1 & 0 \\
0 & 1
\end{pmatrix} 
\bmod{N}
\right\}.
\end{equation*}

\begin{theorem}
\label{them:parabolic-level5}
{\it
The parabolic points of $\Gamma_{\theta,5}$ are given by 
$
\displaystyle
\infty, \, -1, \, 
\pm\frac12, \, 
\pm\frac32, \, 
\pm\frac52. 
$
}
\end{theorem}

\begin{proof}
Considering $\lambda_5$, 
as a  coset decomposition of $\Gamma(1)$ modulo $\Gamma(5),$ 
we may choose 
\begin{align*}
&
\pm 
\begin{pmatrix}
1 & 0 \\
0 & 1
\end{pmatrix}, 
\pm 
\begin{pmatrix}
0 & -1 \\
1 & 0
\end{pmatrix}, 
\pm 
\begin{pmatrix}
1 & 1 \\
0 & 1
\end{pmatrix}, 
\pm 
\begin{pmatrix}
0 & -1 \\
1 & 1
\end{pmatrix}, 
\pm 
\begin{pmatrix}
1 & 2 \\
0 & 1
\end{pmatrix}, 
\pm 
\begin{pmatrix}
0 & -1 \\
1 & 2
\end{pmatrix},   \\
&
\pm 
\begin{pmatrix}
1 & -2 \\
0 & 1
\end{pmatrix}, 
\pm 
\begin{pmatrix}
0 & -1 \\
1 & -2
\end{pmatrix},   
\pm 
\begin{pmatrix}
1 & -1 \\
0 & 1
\end{pmatrix}, 
\pm 
\begin{pmatrix}
0 & -1 \\
1 & -1
\end{pmatrix}, 
\pm 
\begin{pmatrix}
1 & 0 \\
-1 & 1
\end{pmatrix}, 
\pm 
\begin{pmatrix}
1 & -1 \\
1 & 0
\end{pmatrix},   \\
&
\pm 
\begin{pmatrix}
1 & 1 \\
-1 & 0
\end{pmatrix}, 
\pm 
\begin{pmatrix}
1 & 0 \\
1 & 1
\end{pmatrix}, 
\pm 
\begin{pmatrix}
1 & 2 \\
-1 & -1
\end{pmatrix}, 
\pm 
\begin{pmatrix}
1 & 1 \\
1 & 2
\end{pmatrix},   
\pm 
\begin{pmatrix}
1 & -2 \\
-1 & 3
\end{pmatrix}, 
\pm 
\begin{pmatrix}
1 & -3 \\
1 & -2
\end{pmatrix},   \\
&
\pm 
\begin{pmatrix}
1 & -1 \\
-1 & 2
\end{pmatrix}, 
\pm 
\begin{pmatrix}
1 & -2 \\
1 & -1
\end{pmatrix},   
\pm 
\begin{pmatrix}
1 & 0 \\
-2 & 1
\end{pmatrix}, 
\pm 
\begin{pmatrix}
2 & -1 \\
1 & 0
\end{pmatrix},   
\pm 
\begin{pmatrix}
1 & 1 \\
-2 & -1
\end{pmatrix}, 
\pm 
\begin{pmatrix}
2 & 1 \\
1 & 1
\end{pmatrix},   \\
&
\pm 
\begin{pmatrix}
1 & -1 \\
-2 & 3
\end{pmatrix}, 
\pm 
\begin{pmatrix}
2 & -3 \\
1 & -1
\end{pmatrix}, 
\pm 
\begin{pmatrix}
1 & -2 \\
-2 & 5
\end{pmatrix}, 
\pm 
\begin{pmatrix}
2 & -5 \\
1 & -2
\end{pmatrix}, 
\pm 
\begin{pmatrix}
1 & 2 \\
-2 & -3
\end{pmatrix}, 
\pm 
\begin{pmatrix}
2 & 3 \\
1 & 2
\end{pmatrix},   \\
&
\pm 
\begin{pmatrix}
1 & 0 \\
2 & 1
\end{pmatrix}, 
\pm 
\begin{pmatrix}
-2 & -1 \\
1 & 0
\end{pmatrix}, 
\pm 
\begin{pmatrix}
1 & 1 \\
2 & 3
\end{pmatrix}, 
\pm 
\begin{pmatrix}
-2 & -3 \\
1 & 1
\end{pmatrix}, 
\pm 
\begin{pmatrix}
1 & -1 \\
2 & -1
\end{pmatrix}, 
\pm 
\begin{pmatrix}
-2 & 1 \\
1 & -1
\end{pmatrix},   \\
&
\pm 
\begin{pmatrix}
1 & 2 \\
2 & 5
\end{pmatrix}, 
\pm 
\begin{pmatrix}
-2 & -5 \\
1 & 2
\end{pmatrix}, 
\pm 
\begin{pmatrix}
1 & -2 \\
2 & -3
\end{pmatrix}, 
\pm 
\begin{pmatrix}
-2 & 3 \\
1 & -2
\end{pmatrix}, 
\pm 
\begin{pmatrix}
-3 & 1 \\
2 & -1
\end{pmatrix}, 
\pm 
\begin{pmatrix}
-2 & 1 \\
-3 & 1
\end{pmatrix},   \\
&
\pm 
\begin{pmatrix}
-3 & -2 \\
2 & 1
\end{pmatrix}, 
\pm 
\begin{pmatrix}
-2 & -1 \\
-3 & -2
\end{pmatrix}, 
\pm 
\begin{pmatrix}
-3 & -5 \\
2 & 3
\end{pmatrix}, 
\pm 
\begin{pmatrix}
-2 & -3 \\
-3 & -5
\end{pmatrix}, 
\pm 
\begin{pmatrix}
-3 & 2 \\
-2 & 1
\end{pmatrix}, 
\pm 
\begin{pmatrix}
2 & -1 \\
-3 & 2
\end{pmatrix},  \\
&
\pm 
\begin{pmatrix}
-3 & 5 \\
-2 & 3
\end{pmatrix}, 
\pm 
\begin{pmatrix}
2 & -3 \\
-3 & 5
\end{pmatrix}, 
\pm 
\begin{pmatrix}
-5 & 2 \\
2 & -1
\end{pmatrix}, 
\pm 
\begin{pmatrix}
-2 & 1 \\
-5 & 2
\end{pmatrix}, 
\pm 
\begin{pmatrix}
-5 & 7 \\
2 & -3
\end{pmatrix}, 
\pm 
\begin{pmatrix}
-2 & 3 \\
-5 & 7
\end{pmatrix},   \\
&
\pm 
\begin{pmatrix}
5 & 12 \\
2 & 5
\end{pmatrix}, 
\pm 
\begin{pmatrix}
-2 & -5 \\
5 & 12
\end{pmatrix},   
\pm 
\begin{pmatrix}
5 & 2 \\
2 & 1
\end{pmatrix}, 
\pm 
\begin{pmatrix}
-2 & -1 \\
5 & 2
\end{pmatrix},   
\pm 
\begin{pmatrix}
5 & 7 \\
2 & 3
\end{pmatrix}, 
\pm 
\begin{pmatrix}
-2 & -3 \\
5 & 7
\end{pmatrix},   
\end{align*}
which implies that 
\begin{align*}
\Gamma(1)=&
\Gamma_{\theta,5} 
\cup
\Gamma_{\theta,5} 
\begin{pmatrix}
1 & 1 \\
0 & 1
\end{pmatrix} 
\cup
\Gamma_{\theta,5} 
\begin{pmatrix}
1 & -1 \\
0 & 1
\end{pmatrix} 
\cup
\Gamma_{\theta,5} 
\begin{pmatrix}
1 & 2 \\
0 & 1
\end{pmatrix} 
\cup
\Gamma_{\theta,5} 
\begin{pmatrix}
1 & -2 \\
0 & 1
\end{pmatrix}   \\
&
\cup
\Gamma_{\theta,5} 
\begin{pmatrix}
1 & 0 \\
-1 & 1
\end{pmatrix} 
\cup
\Gamma_{\theta,5} 
\begin{pmatrix}
1 & 1 \\
-1 & 0
\end{pmatrix} 
\cup
\Gamma_{\theta,5} 
\begin{pmatrix}
1 & 2 \\
-1 & -1
\end{pmatrix} 
\cup
\Gamma_{\theta,5} 
\begin{pmatrix}
1 & -2 \\
-1 & 3
\end{pmatrix} 
\cup
\Gamma_{\theta,5} 
\begin{pmatrix}
1 & -1 \\
-1 & 2
\end{pmatrix}   \\
&
\cup
\Gamma_{\theta,5} 
\begin{pmatrix}
1 & 0 \\
-2 & 1
\end{pmatrix}   
\cup
\Gamma_{\theta,5} 
\begin{pmatrix}
1 & 1 \\
-2 & -1
\end{pmatrix}   
\cup
\Gamma_{\theta,5} 
\begin{pmatrix}
1 & -1 \\
-2 & 3
\end{pmatrix}   
\cup
\Gamma_{\theta,5} 
\begin{pmatrix}
1 & -2 \\
-2 & 5
\end{pmatrix}   
\cup
\Gamma_{\theta,5} 
\begin{pmatrix}
1 & 2 \\
-2 & -3
\end{pmatrix}     \\
&
\cup
\Gamma_{\theta,5} 
\begin{pmatrix}
1 & 0 \\
2 & 1
\end{pmatrix}   
\cup
\Gamma_{\theta,5} 
\begin{pmatrix}
1 & 1 \\
2 & 3
\end{pmatrix}   
\cup
\Gamma_{\theta,5} 
\begin{pmatrix}
1 & -1 \\
2 & -1
\end{pmatrix}   
\cup
\Gamma_{\theta,5} 
\begin{pmatrix}
1 & 2 \\
2 & 5
\end{pmatrix}   
\cup
\Gamma_{\theta,5} 
\begin{pmatrix}
1 & -2 \\
2 & -3
\end{pmatrix}     \\
&
\cup
\Gamma_{\theta,5} 
\begin{pmatrix}
-3 & 1 \\
2 & -1
\end{pmatrix}  
\cup
\Gamma_{\theta,5} 
\begin{pmatrix}
-3 & -2 \\
2 & 1
\end{pmatrix}  
\cup
\Gamma_{\theta,5} 
\begin{pmatrix}
-3 & -5 \\
2 & 3
\end{pmatrix}  \\
&
\cup
\Gamma_{\theta,5} 
\begin{pmatrix}
-3 & 2 \\
-2 & 1
\end{pmatrix}  
\cup
\Gamma_{\theta,5} 
\begin{pmatrix}
-3 & 5 \\
-2 & 3
\end{pmatrix}    \\
&
\cup
\Gamma_{\theta,5} 
\begin{pmatrix}
-5 & 2 \\
2 & 1
\end{pmatrix}  
\cup
\Gamma_{\theta,5} 
\begin{pmatrix}
-5 & 7 \\
2 & -3
\end{pmatrix}  \\
&
\cup
\Gamma_{\theta,5} 
\begin{pmatrix}
5 & 12 \\
2 & 5
\end{pmatrix}
\cup
\Gamma_{\theta,5} 
\begin{pmatrix}
5 & 2 \\
2 & 1
\end{pmatrix}
\cup
\Gamma_{\theta,5} 
\begin{pmatrix}
5 & 7 \\
2 & 3
\end{pmatrix}.
\end{align*}
Moreover, we obtain 
\begin{align*}
\Gamma(1)=&
\Gamma_{\theta,5} 
\cup
\Gamma_{\theta,5} 
S
\cup
\Gamma_{\theta,5} 
S^{-1}
\cup
\Gamma_{\theta,5} 
S^2
\cup
\Gamma_{\theta,5} 
S^{-2}   \\
&\cup
\Gamma_{\theta,5}  S^{-1}T^{-1}S^{-1}
\cup
\Gamma_{\theta,5} 
S^{-1}T^{-1}
\cup
\Gamma_{\theta,5} 
S^{-1}T^{-1}S
\cup
\Gamma_{\theta,5} 
S^{-1}T^{-1}S^2
\cup
\Gamma_{\theta,5} 
S^{-1}T^{-1}S^{-2}  \\
&\cup
\Gamma_{\theta,5}  
S^2T^{-1}
\cup
\Gamma_{\theta,5} 
S^2T^{-1}S
\cup
\Gamma_{\theta,5} 
S^2T^{-1}S^{-1}
\cup
\Gamma_{\theta,5} 
S^2T^{-1}S^{-2}
\cup
\Gamma_{\theta,5} 
S^2T^{-1}S^2  \\
&\cup
\Gamma_{\theta,5}  
S^{-2} T^{-1}
\cup
\Gamma_{\theta,5} 
S^{-2}T^{-1} S
\cup
\Gamma_{\theta,5} 
S^{-2}T^{-1}S^{-1}
\cup
\Gamma_{\theta,5} 
S^{-2}T^{-1}S^2
\cup
\Gamma_{\theta,5} 
S^{-2}T^{-1} S^{-2}     \\
&\cup
\Gamma_{\theta,5}  
S^{-2}T^{-1}S^{-2}T^{-1}S^{-1}
\cup
\Gamma_{\theta,5} 
S^{-2}T^{-1}S^{-2}T^{-1}
\cup
\Gamma_{\theta,5} 
S^{-2}T^{-1}S^{-2}T^{-1}S   \\
&\cup
\Gamma_{\theta,5}  
S^2T^{-1}S^2T^{-1}
\cup
\Gamma_{\theta,5} 
S^2T^{-1}S^2T^{-1}S^{-1}   \\
&\cup
\Gamma_{\theta,5}  
S^2T^{-1}S^{-2}T^{-1}S^{-1}
\cup
\Gamma_{\theta,5} 
S^2T^{-1}S^{-2}T^{-1}S^{-2} \\
&\cup
\Gamma_{\theta,5}  
S^2T^{-1}S^{-2}T^{-1}S^2
\cup
\Gamma_{\theta,5} 
S^2T^{-1}S^{-2}T^{-1}
\cup
\Gamma_{\theta,5} 
S^2T^{-1}S^{-2}T^{-1}S,
\end{align*}
where 
\begin{equation*}
S=
\begin{pmatrix}
1 & 1 \\
0 & 1
\end{pmatrix}, 
\,\,
T=
\begin{pmatrix}
0& -1 \\
1 & 0
\end{pmatrix}. 
\end{equation*}
\end{proof}

\begin{theorem}
{\it
An elliptic element of $\Gamma_{\theta,5}$ is conjugate to 
\begin{equation*}
\pm 
\begin{pmatrix}
0 & -1 \\
1 & 0
\end{pmatrix}, 
\pm 
\begin{pmatrix}
70 & -109 \\
29 & -70
\end{pmatrix}.
\end{equation*}
}
\end{theorem}

\begin{proof}
Recall that 
an elliptic element of $\Gamma_{\theta,5}$ is conjugate to 
\begin{equation*}
\pm
M
\begin{pmatrix}
0 & -1 \\
1 & 0
\end{pmatrix}
M^{-1}, 
\pm
M
\begin{pmatrix}
0 & 1 \\
-1 & -1
\end{pmatrix}
M^{-1}, 
\pm
M
\begin{pmatrix}
-1 & -1 \\
1 & 0
\end{pmatrix}
M^{-1},
\end{equation*}
where $M$ is a coset representative in the proof of Theorem \ref{them:parabolic-level5}. 
We obtain 
\begin{equation*}
\pm 
\begin{pmatrix}
0 & -1 \\
1 & 0
\end{pmatrix}, 
\pm 
\begin{pmatrix}
70 & -109 \\
29 & -70
\end{pmatrix}.
\end{equation*}
\par
Suppose that the matrices of the theorem are equivalent, which implies that 
there exists 
$
\begin{pmatrix}
a & b \\
c & d
\end{pmatrix}
\in
\Gamma_{\theta,5}
$ 
such that 
\begin{equation*}
\pm
\begin{pmatrix}
a & b \\
c & d
\end{pmatrix}
\begin{pmatrix}
0 & -1 \\
1 & 0
\end{pmatrix}
=
\begin{pmatrix}
70 & -109 \\
29 & -70
\end{pmatrix}
\begin{pmatrix}
a & b \\
c & d
\end{pmatrix}. 
\end{equation*}
\par
If 
$
\begin{pmatrix}
a & b \\
c & d
\end{pmatrix}
\begin{pmatrix}
0 & -1 \\
1 & 0
\end{pmatrix}
=
\begin{pmatrix}
70 & -109 \\
29 & -70
\end{pmatrix}
\begin{pmatrix}
a & b \\
c & d
\end{pmatrix}, 
$ 
we have 
\begin{equation*}
70a-b\equiv a+70b \equiv 0 \bmod{169},
\end{equation*}
and
\begin{equation*}
c=\frac{70a-b}{169}, d=\frac{a+70b}{169}.
\end{equation*}
Since $a\equiv d \bmod{5}$ and $b\equiv-c \bmod{5},$ 
it follows that 
$a\equiv b\equiv 0 \bmod{5},$ which is impossible. 
\par
If 
$
-
\begin{pmatrix}
a & b \\
c & d
\end{pmatrix}
\begin{pmatrix}
0 & -1 \\
1 & 0
\end{pmatrix}
=
\begin{pmatrix}
70 & -109 \\
29 & -70
\end{pmatrix}
\begin{pmatrix}
a & b \\
c & d
\end{pmatrix}, 
$ 
we have 
\begin{equation*}
70a+b\equiv -a+70b \equiv 0 \bmod{169},
\end{equation*}
and
\begin{equation*}
c=\frac{70a+b}{169}, d=\frac{-a+70b}{169}.
\end{equation*}
Since $ad-bc=1,$ 
it follows that 
\begin{equation*}
a^2+b^2=-169,
\end{equation*}
which is impossible. 
\end{proof}




\end{document}